\documentclass[11pt]{article}

\usepackage{amsmath}
\usepackage{amssymb}
\usepackage{amsthm}
\usepackage{mathtools}
\usepackage{bm}

\usepackage{geometry}
\usepackage{graphicx}
\usepackage{subcaption}
\usepackage{xcolor}
\usepackage{enumitem}
\usepackage{booktabs}
\usepackage{titlesec}
\usepackage[T1]{fontenc}
\usepackage{lmodern}

\usepackage{tcolorbox}
\tcbuselibrary{breakable,skins}

\usepackage{csquotes}
\usepackage[backend=biber,style=alphabetic]{biblatex}

\usepackage{hyperref}
\usepackage{cleveref}
\ifdefined\AddToHook
  \AddToHook{cmd/appendix/after}{%
    \crefalias{section}{appendix}%
    \crefalias{subsection}{subappendix}%
    \crefalias{subsubsection}{subsubappendix}%
  }
\fi

\numberwithin{equation}{section}
\allowdisplaybreaks

\usepackage{autonum}
\makeatletter
\ifx\forlistloop\autonum@etoolbox@forlistloop
  \def\blx@noerroretextools{}
\fi
\makeatother
\graphicspath{{picture_files/main_final/}}
\let\Re\relax\DeclareMathOperator{\Re}{Re}
\let\Im\relax\DeclareMathOperator{\Im}{Im}

\DeclarePairedDelimiter{\abs}{\lvert}{\rvert}
\DeclarePairedDelimiter{\norm}{\lVert}{\rVert}
\DeclarePairedDelimiter{\rbra}{\lparen}{\rparen}
\DeclarePairedDelimiter{\cbra}{\lbrace}{\rbrace}
\DeclarePairedDelimiter{\sbra}{\lbrack}{\rbrack}
\DeclarePairedDelimiterX{\inner}[2]{\langle}{\rangle}{#1, #2}
\DeclarePairedDelimiterX{\Set}[2]{\lbrace}{\rbrace}{#1\,\colon\,#2}

\DeclarePairedDelimiterXPP\Prob[2]{#1}{\lparen}{\rparen}{}{#2}
\DeclarePairedDelimiterXPP\Exp[2]{\mathbb{E}_{#1}}{\lbrack}{\rbrack}{}{#2}

\theoremstyle{plain}
\newtheorem{theorem}{Theorem}[section]
\newtheorem{proposition}{Proposition}[section]
\newtheorem{lemma}{Lemma}[section]
\newtheorem{corollary}{Corollary}[section]

\theoremstyle{definition}
\newtheorem{definition}{Definition}[section]

\theoremstyle{remark}
\newtheorem{remark}{Remark}[section]
\newtheorem{example}{Example}[section]

\newtheorem*{theorem*}{Theorem}
\newtheorem*{proposition*}{Proposition}
\newtheorem*{lemma*}{Lemma}
\newtheorem*{corollary*}{Corollary}
\newtheorem*{definition*}{Definition}
\newtheorem*{assumption*}{Assumption}
\newtheorem*{remark*}{Remark}
\newtheorem*{example*}{Example}

\crefname{theorem}{Theorem}{Theorems}
\crefname{proposition}{Proposition}{Propositions}
\crefname{lemma}{Lemma}{Lemmas}
\crefname{corollary}{Corollary}{Corollaries}
\crefname{definition}{Definition}{Definitions}
\crefname{assumption}{Assumption}{Assumptions}
\crefname{remark}{Remark}{Remarks}
\crefname{example}{Example}{Examples}
\crefname{equation}{Equation}{Equations}
\crefname{figure}{Figure}{Figures}
\crefname{table}{Table}{Tables}
\crefname{section}{Section}{Sections}
\crefname{subsection}{Subsection}{Subsections}
\crefname{appendix}{Appendix}{Appendices}
\Crefname{appendix}{Appendix}{Appendices}

\input{settings/layout_en.tex}
\DeclareDelimFormat[bib]{nametitledelim}{\addcolon\space}

\renewbibmacro{in:}{}

\renewbibmacro*{volume+number+eid}{
    \printfield{volume}
    \setunit*{\addnbspace}
    \printfield{number}
    \setunit{\addcomma\space}
    \printfield{eid}
}
\DeclareFieldFormat[article]{volume}{\mkbibbold{#1}}
\DeclareFieldFormat[article]{number}{\mkbibparens{#1}}

\renewbibmacro*{issue+date}{}
\renewbibmacro*{note+pages}{
    \printfield{note}
    \setunit{\bibpagespunct}
    \printfield{pages}
    \setunit{\addcomma\addspace}
    \usebibmacro{date}
    \newunit
}
\DeclareFieldFormat{pages}{#1}
\DeclareFieldFormat{date}{\mkbibparens{#1}}

\DeclareFieldFormat{parenswithperiod}{\mkbibbrackets{#1\addperiod}}

\renewbibmacro*{pageref}{
    \iflistundef{pageref}{}{
        \printtext[parenswithperiod]{
            \ifnumgreater{\value{pageref}}{1}
                {\bibstring{backrefpages}\ppspace}
                {\bibstring{backrefpage}\ppspace}
            \printlist[pageref][-\value{listtotal}]{pageref}
        }
        
    }
}
\DefineBibliographyStrings{english}{
    backrefpage  = {\hspace{-0.8mm}cit.\,on page\hspace{-1.2mm}},
    backrefpages = {\hspace{-0.8mm}cit.\, on pages\hspace{-1.2mm}},
}
\DeclareCiteCommand{\citeauthor}{
    \boolfalse{citetracker}%
    \boolfalse{pagetracker}%
    \usebibmacro{prenote}
    }{\ifciteindex
        {\indexnames{labelname}}
        {}%
    \printtext[bibhyperref]{\printnames{labelname}\!}
    }{\multicitedelim}{\usebibmacro{postnote}
}
\DeclareCiteCommand{\citeauthoryear}{
    \boolfalse{citetracker}%
    \boolfalse{pagetracker}%
    \usebibmacro{prenote}
    }{\ifciteindex
        {\indexnames{labelname}}
        {}%
    \printtext[bibhyperref]{\printnames{labelname}}\printtext{\hspace{1mm}(\printfield{year})\hspace{-2mm}}
    }{\multicitedelim}{\usebibmacro{postnote}
}

\begin{document}

\title{A limit theorem for generalized tempered stable processes with stable index tending to 2 and its application to finance}
\author{Masaaki Fukasawa \and Mikio Hirokane\thanks{Graduate School of Engineering Science, The University of Osaka, 1-3, Toyonaka-shi, Osaka, Japan. E-mail: \texttt{hirokane@sigmath.es.osaka-u.ac.jp}}}
\date{18 November 2026}
\maketitle

\begin{abstract}
    We investigate the joint limiting behavior of a multidimensional Generalized Tempered Stable (GTS) process and its quadratic covariation as the stable index tends to two.
    Under appropriate scaling, the GTS process converges to a Brownian motion, while its renormalized covariation process converges to an independent stable process with stable index one.
    As an application, we construct a pure-jump perturbation model around the Black-Scholes model and derive an asymptotic expansion of the at-the-money implied volatility skew.
\end{abstract}



\tableofcontents

\section{Introduction}

The Black--Scholes implied volatility refers to the volatility input to the Black--Scholes formula that reproduces observed option prices.
While the Black--Scholes model yields a constant implied volatility, observed implied volatilities often exhibit notable dependence on both strike price and maturity.  
For a fixed maturity, the curve of the implied volatility as a function of strike is typically non-flat, forming what is known as the volatility smile or skew.  
In particular, it is well documented that the at-the-money (ATM) implied volatility skew exhibits a power-law decay with respect to maturity;
see, for example, \cite{alos2007shorttime}, \cite{fouque2004maturity}, and \cite{gatheral2018volatility}.  
It is therefore of both theoretical and practical importance to investigate which mathematical structures of asset price models are essential for capturing such observed behavior.

Outside the Black--Scholes model, implied volatilities generally do not admit closed-form expressions, which motivates the study of their asymptotic expansions.  
There are two major directions in this line of research.  
The first is the short-time asymptotic approach, where we analyze the behavior of the implied volatility as the maturity tends to zero under a fixed asset price model.  
For example, Gerhold et al.~\cite{gerhold2016smallmaturity} and Figueroa-L\'opez and \'{O}lafsson~\cite{figueroa-lopez2016shortterm} consider models involving $\alpha$-stable-type jumps and show that the ATM implied volatility skew decays according to a power law in maturity.
Empirical studies often describe the observed ATM implied volatility skew as having a power-law-type dependence on maturity~\cite{gatheral2018volatility}, with short-maturity behavior in some cases close to \(t^{-1/2}\)~\cite{fouque2004maturity}, and the rate \(t^{(1-\alpha)/2}\) arising in such stable-type settings is consistent with this type of steep short-maturity behavior when $\alpha$ is close to $2$.
This indicates that a class of models with stable-type jumps is one candidate for capturing the observed implied volatility behavior.
More broadly, Andersen and Lipton~\cite{andersen2013asymptotics} study implied-volatility asymptotics for exponential L\'evy models, including tempered stable classes, for a fixed model.
The second approach is based on perturbation methods, where a family of models converging to the Black--Scholes model is constructed; see, for example, \cite{lewis2000option}, \cite{fouque2003singular}, and \cite{fukasawa2011asymptotic}.
The common feature of these studies is that they derive asymptotic corrections to option prices or implied volatilities around the Black--Scholes model for a prescribed perturbative family. 
To the best of our knowledge, such perturbative analyses have not been developed for stable-type infinite-activity pure-jump families approaching the Black--Scholes model as a parameter varies. 
Our starting point is that a Brownian motion is an $\alpha$-stable process with $\alpha = 2$.
We extend this viewpoint to generalized tempered stable (GTS) processes and thereby obtain a pure-jump family approaching the Black--Scholes model through stable-type infinite-activity jumps.
This perspective is also consistent with empirical evidence from high-frequency data suggesting stable-like small-jump activity in asset returns; see \cite{ait-sahalia2009estimating,ait-sahalia2012analyzing}.

In this study, we focus on pure-jump models driven by GTS processes, which generalize stable processes by tempering their L\'evy measures~\cite{rosinski2010generalized}.  
Since a stable process becomes a Brownian motion when the stable index~$\alpha = 2$, we investigate the limiting behavior of GTS processes and their quadratic covariations as $\alpha \uparrow 2$.  
Under a proper scaling of the L\'evy measure, we show that the GTS process converges in distribution to a Brownian motion.  
Simultaneously, we renormalize the quadratic covariation process so that it has a non-degenerate limit, which is a stable process with stable index 1, independent of the limiting Brownian motion.

This result should be distinguished from the literature on Gaussian approximations for small-jump L\'evy processes.
A classical idea is that the cumulative contribution of many small jumps can be approximated by a Gaussian component; see, e.g., \cite{asmussen2001approximations} and, in the multivariate stable and tempered stable setting, \cite{cohen2007gaussian}.
In these studies, Gaussian approximations are obtained by truncating the jump measure so as to isolate the small-jump component and thereby restore square-integrability, allowing the use of central-limit-type arguments.
K\"uchler and Tappe~\cite{kuchler2013tempered} study tempered stable distributions and processes.
In their setting, tempering ensures integrability properties that allow Gaussian approximation results within a fixed stable-index range (corresponding to \(\alpha \in (0,1)\) in our notation).
Their analysis therefore concerns tempered models in a regime where classical Gaussian approximation techniques apply.
By contrast, we consider a family of GTS processes indexed by \(\alpha\), which includes pure stable processes, and study the limit as \(\alpha \uparrow 2\).
Our result does not rely on truncation or on a central limit theorem for square-integrable variables.
Instead, it describes the limit of an accumulation of jumps that lies outside the usual square-integrable setting for central limit arguments.
The Brownian limit thus emerges from a singular limit of the model family itself: as \(\alpha \uparrow 2\), the L\'evy measure becomes increasingly concentrated near the origin and the small-jump behavior becomes more singular rather than more integrable.

As an application, we construct a pure-jump exponential L\'evy model based on the scaled GTS process, interpreted as a perturbation of the Black--Scholes model.  
We then derive an asymptotic expansion of the ATM implied volatility skew in this framework.  
Notably, the limit of the normalized quadratic covariation is not integrable, which prevents the direct application of the martingale expansion theory developed by Yoshida~\cite{yoshida2001malliavin}.  
Instead, we analyze the characteristic function and apply the Carr--Madan formula~\cite{carr1999option} to derive an asymptotic expansion of the implied volatility.
This yields an explicit characterization of the first-order correction to the ATM skew when a family of GTS models approaches the Black--Scholes model.
The first-order term is identified for arbitrary maturities, and we further determine its exact short-maturity behavior.
In particular, our result produces the exact $t^{-1/2}$ term structure in the leading term of the skew expansion.
For comparison, several perturbative studies around the Black--Scholes model in fractional volatility settings lead to different maturity scalings.
Fukasawa~\cite{fukasawa2011asymptotic} considers a small volatility-of-volatility perturbation of a fractional volatility model with Hurst index $H>0$ and obtains an ATM skew behaving as $t^{H-1/2}$.
Garnier and S\o lna~\cite[Section 6]{garnier2018option} discuss a fast mean-reverting perturbation in a fractional volatility model, in which the ATM skew scales as \(t^{H-3/2}\) for \(H\in(1/2,1)\).
In a related small volatility-of-volatility regime, Bayer et al.~\cite{bayer2021logmodulateda} obtain a first-order skew correction of the form $\varepsilon\log (1/t)^{-p}t^{H-1/2}$, which in particular is not a pure power law in $t$.
We are not aware of perturbative expansions around the Black--Scholes model in which the first-order correction to the ATM skew exhibits the exact short-maturity scaling $t^{-1/2}$.
To the best of our knowledge, our result provides the first example in which such a $t^{-1/2}$ structure arises from a pure-jump perturbation approaching the Black--Scholes model.
The result should therefore be viewed primarily as a theoretical characterization of how this canonical skew scaling can emerge from a singular pure-jump limit, rather than as a calibration tool.

\section{GTS processes}

We begin by recalling key properties of L\'evy and stable processes that are essential for our main results.
Let $(\Omega,\mathcal{F},\mathbb{P})$ be a probability space.
According to Akritas \cite{akritas1982asymptotic} and Jacod and Shiryaev~\cite[Chapter II, 4.26 Proposition]{jacod2003limit},
we know that the quadratic variation of a L\'evy process is a L\'evy process, and
how to compute the L\'evy measure of the quadratic variation. Using the L\'evy--It\^o decomposition \cite[Theorem 19.2]{ken-iti1999levy},
here we determine the joint distribution of a L\'evy process and its quadratic covariation process.
Let $\inner*{A}{B}\coloneqq\mathrm{tr}(A^{\top}B)$ for $A,B\in\mathbb{R}^{d\times d}$.
For completeness, we include a proof in Appendix A.

\begin{lemma}[The joint distribution of a L\'evy process and its quadratic covariation process]\label{le:jointdist}
    Let $X$ be an $\mathbb{R}^d$-valued L\'evy process with triplet $(m,V,\nu)$ and $[X]$ be the optional quadratic covariation process of $X$.
    Then the $\mathbb{R}^d\times\mathbb{R}^{d\times d}$-valued L\'evy process $(X,[X])$ has the following characteristic function:
    \begin{align}
        &\Exp*{}{e^{i(\inner{z}{X_1}+\inner{U}{[X]_1})}} \\
        &\hspace{7mm}=\exp\cbra*{i\inner{m}{z}-\frac{1}{2}\inner{z}{Vz}+i\inner{U}{V}+\int_{\mathbb{R}^d}(e^{i(\inner{z}{x}+\inner{x}{Ux})}-1-i\inner{z}{x}\bm{1}_{\left\{|x|\le 1\right\}})\nu(d x)},
    \end{align}
    where $z\in\mathbb{R}^d$ and $U\in\mathbb{R}^{d\times d}$.
\end{lemma}

A stable process is a L\'evy process whose distribution at time 1 is stable.
It is characterized by the \emph{stable index} $\alpha\in(0,2]$, and, for $\alpha<2$, its L\'evy measure takes the form
\begin{equation}
    \nu(B)=\int_S\rho(d\xi)\int_0^\infty \bm{1}_{B}(r\xi)\cdot\frac{1}{r^{\alpha +1}}d r,\qquad B\in\mathcal{B}(\mathbb{R}^d\backslash\left\{0\right\}),
\end{equation}
where $S$ is the unit sphere and $\rho$ is a finite measure on $S$ representing the distribution of jump directions; see \cite[Theorem 14.3]{ken-iti1999levy}.

Given the jump direction $\xi$, the stable process has only two parameters $(\rho(\xi),\alpha)$ to characterize jump size distribution.
To add flexibility, a GTS process was introduced by Rosi\'nski and Sinclair~\cite{rosinski2010generalized}.

\begin{definition}[GTS process]\label{def:gts}
    Let $X$ be an $\mathbb{R}^d$-valued L\'evy process with triplet $(m,V,\nu)$.
    $X$ is called a \emph{GTS process with} $(m,\alpha,\rho,q)$
    if $V=0$ and the L\'evy measure $\nu$ satisfies
    \begin{equation}
        \nu(B)=\int_S\rho(d\xi)\int_0^\infty \bm{1}_{B}(r\xi)\cdot\frac{q(r,\xi)}{r^{\alpha +1}}d r
    \end{equation}
    for any $B\in\mathcal{B}(\mathbb{R}^d\backslash\left\{0\right\})$,
    where $\alpha\in(0,2)$, $\rho$ is a finite measure on $S\coloneqq \Set{\xi\in\mathbb{R}^d}{|\xi|=1}$, and 
    a measurable function $q:(0,\infty)\times S\to [0,\infty)$ has a $\rho$-integrable function $p:S\to [0,\infty)$
    with $\int_S \abs*{q(r,\xi)-p(\xi)}\rho(d\xi)$ $\to 0$ as $r\downarrow 0$.
\end{definition}
For a GTS process, we assume the following technical conditions.
\begin{enumerate}
    \item[(G-1)] \;The function $q$ is bounded by a constant $M>0$.
    \item[(G-2)] \;For $\rho$ almost every (a.e.) $\xi\in S$, $q(0+,\xi) \coloneqq \lim_{r \downarrow 0} q(r,\xi)$ exists and is finite.
    \item[(G-3)] \;We have $\int_S\rho(d\xi)\int_0^1 \abs*{q(r,\xi)- q(0+,\xi)}r^{-1}d r<\infty$.
\end{enumerate}

The tempering function $q$ modulates the jump intensity in each direction.
If $q$ is a constant, the GTS process is a stable process.
Thus, a GTS process is fully characterized by the four components \( (m, \alpha, \rho, q) \),
and its characteristic function is written as
\begin{align}
    \Exp{}{e^{i \inner{z}{X_1}}}
    &=\exp\cbra*{i\inner{m}{z}+
    \int_S\rho(d\xi)\int_0^\infty(e^{i\inner{z}{r\xi}}-1-i\inner{z}{r\xi} \bm{1}_{(0,1)}(r))\cdot\frac{q(r,\xi)}{r^{\alpha +1}}d r}.
\end{align}

\begin{example}[The Carr--Geman--Madan--Yor (CGMY) process]
    If an $\mathbb{R}$-valued GTS process has the components $\rho(\{1\})\coloneqq C$, $\rho(\{-1\})\coloneqq C$, $q(r,1)\coloneqq e^{-Mr}$, $q(r,-1)\coloneqq e^{-Gr}$, and $\alpha\coloneqq Y$,
    where $C>0$, $G\ge 0$, $M\ge 0$, and $0<Y<2$,
    this is called the \emph{CGMY process}; see \cite{carr2002fine}.
\end{example}

Among assumptions (G-1)--(G-3), the substantive additional regularity condition is (G-3), which is used later to control the error from replacing $q(r,\xi)$ by $q(0+,\xi)$ near $r=0$.
It is satisfied, for example, if $q(\cdot,\xi)$ is locally H\"older continuous at $0$ uniformly in $\xi$ outside a $\rho$-null set; in particular, standard exponentially tempered examples such as the CGMY processes satisfy it.
By contrast, (G-3) may fail if $q(r,\xi)$ oscillates near the origin or converges too slowly to its limit.

\section{Limit theorem for the scaled GTS process}

Since a stable process reduces to Brownian motion when $\alpha = 2$,
it is natural to expect that a GTS process converges to Brownian motion as $\alpha \uparrow 2$.
However, recalling
\begin{equation}\label{eq:gamma}
    \int_0^\infty (e^{i x r} - 1 - i xr) \frac{1}{r^{\alpha +1}}d r
    = \Gamma(-\alpha) |x|^\alpha e^{-i \frac{\pi \alpha}{2}\mathrm{sgn}(x)},\qquad x\in\mathbb{R},
\end{equation}
where $\Gamma$ is the gamma function (see the proof of Theorem 14.10 in \cite{ken-iti1999levy}) and 
$\Gamma(-\alpha) \sim 1/(4-2\alpha) \to \infty$ as $\alpha \uparrow 2$,
we find the L\'evy measure would diverge without a proper scaling.
To ensure convergence, we introduce the following scaling.

\begin{definition}[Scaled GTS process]\label{def:scaledgts}
    Let $h:(3/2,2)\to[0,\infty)$ be a function satisfying the limit $h(\alpha)/(2-\alpha)\to 1$ as $\alpha\uparrow 2$ and $X$ be a GTS process with $(m,\alpha,\rho,q)$
    satisfying (G-1), (G-2), and (G-3).
    For each $\alpha\in (3/2,2)$, we define a \emph{scaled GTS process} $X^\alpha$ to be the GTS process with $(m,\alpha,\rho,h(\alpha)q)$.
    The L\'evy measure $\nu_\alpha$ of the scaled GTS process is written as
    \begin{align}
        \nu_\alpha(B)\coloneqq h(\alpha)\int_S\rho(d\xi)\int_0^\infty \bm{1}_{B}(r\xi)\cdot\frac{q(r,\xi)}{r^{\alpha +1}}d r.
    \end{align}
    In addition, let $\Sigma\in\mathbb{R}^{d\times d}$ and $\gamma\in\mathbb{R}$ be
    \begin{equation}
        \Sigma\coloneqq \int_{S}\xi\xi^{\top}q(0+,\xi)\rho(d\xi),\quad 
        \gamma\coloneqq \int_0^\infty(\sin r-r\bm{1}_{(0,1)}(r))\frac{1}{r^2}d r.
    \end{equation}
    The constant $\gamma$ is the standard compensation term; see \cite[Lemma 14.11]{ken-iti1999levy}.
\end{definition}

As the scaled GTS process converges to a Brownian motion, its quadratic covariation process converges to that of the Brownian motion, which is deterministic.  
To obtain a non-degenerate limit, we introduce a suitable normalization of the quadratic covariation.

\begin{definition}[Normalized quadratic covariation]\label{ass:scale}
    Let $c:(3/2,2)\to(0,\infty)$ be a function satisfying the limit $c(\alpha)/(2-\alpha)\to \frac{1}{2}$ as $\alpha\uparrow 2$
    and $D(\alpha)\in\mathbb{R}^{d\times d}$ be
    \begin{equation}
        D(\alpha)\coloneqq \frac{h(\alpha)c(\alpha)^{1-\alpha/2}}{2}
        \int_S \rho (d\xi)\int_0^1
        \xi\xi^{\top}\frac{q(r^{1/2}c(\alpha)^{1/2},\xi)}{r^{\alpha/2}}d r.
    \end{equation}
    We define the \emph{normalized quadratic covariation process} $Y^\alpha$ to be
    \begin{equation}
        Y^\alpha_t\coloneqq\frac{[X^\alpha]_t -D(\alpha)\cdot t}{c(\alpha)}
    \end{equation}
    for each $t\ge 0$, where $[X^\alpha]$ is the optional quadratic covariation of the scaled GTS process $X^\alpha$.
\end{definition}

Here $D(\alpha)$ is the centering term, and $c(\alpha)$ is the corresponding normalizing factor.
The process $Y^\alpha$ represents the renormalized fluctuation of the quadratic covariation.

\begin{theorem}[Main result]\label{th:main_infinite}
    Suppose that $X^\alpha$ is a scaled GTS process
    and $Y^\alpha$ is the normalized quadratic covariation process of $X^\alpha$.
    Let $(X,Y)$ be an $\mathbb{R}^d\times\mathbb{R}^{d\times d}$-valued L\'evy process with
    characteristic function at time $t=1$ given by
    \begin{align}
        &\varphi(z,U)\coloneqq\exp\cbra*{i \inner{z}{m}-\frac{1}{2}\inner{z}{\Sigma z}}\\
        &\hspace{15mm}\times\exp\cbra*{\int_S\rbra*{-\frac{\pi}{2}\abs*{\inner*{\xi}{U\xi}}-i \inner*{\xi}{U\xi}\log\abs*{\inner*{\xi}{U\xi}}+i \gamma \inner*{\xi}{U\xi}}q(0+,\xi)\rho(d\xi)}
    \end{align}
    for any $z\in\mathbb{R}^d$ and $U\in\mathbb{R}^{d\times d}$. Then $(X^\alpha,Y^\alpha) \xrightarrow{\mathrm{d}} (X,Y)$ holds as $\alpha\uparrow 2$
    , where $\xrightarrow{\mathrm{d}}$ denotes convergence in distribution on the Skorokhod space $D(\mathbb{R}^d\times\mathbb{R}^{d\times d})$.
\end{theorem}

The theorem identifies the joint limiting behavior of the process $X^\alpha$ and the renormalized quadratic covariation $Y^\alpha$.
In particular, while $X^\alpha$ has a Gaussian limit, the corresponding fluctuation of the quadratic covariation converges to a $1$-stable limit independent of the limiting Brownian motion.

\begin{proof}[Proof of Theorem \ref{th:main_infinite}]
We work at the level of characteristic functions.
We first identify the Gaussian limit of $X^\alpha$, then derive the $1$-stable limit of the renormalized quadratic covariation by Lemma 2.1.
Finally, we compute the joint characteristic exponent and show that the mixed term vanishes in the limit, which yields the claimed asymptotic independence.

By \cite[Chapter VII, 3.6 Corollary]{jacod2003limit}, it suffices to prove the convergence of the finite-dimensional distributions.
We first show that $X_1^\alpha\xrightarrow{\mathrm{d}} \mathrm{N}(m,\Sigma)$.
The cumulant of $X^\alpha_1$ can be written as follows:
\begin{align}
    &\log \Exp{}{e^{i \inner{z}{X^{\alpha}_1}}}\\
    &=i\inner{m}{z}+h(\alpha)\int_S\rho(d\xi)\int_0^\infty i\inner*{z}{r\xi}\bm{1}_{[1,\infty)}\cdot\frac{q(0+,\xi)}{r^{\alpha+1}}d r\\
    &\hspace{5mm} +h(\alpha)\int_S\rho(d\xi)\int_0^\infty(e^{i\inner{z}{r\xi}}-1-i\inner{z}{r\xi})\cdot\frac{q(0+,\xi)}{r^{\alpha +1}}d r\\
    &\hspace{5mm} +h(\alpha)\int_S\rho(d\xi)\int_0^\infty(e^{i\inner{z}{r\xi}}-1-i\inner{z}{r\xi}\bm{1}_{(0,1)}(r))\cdot\frac{q(r,\xi)-q(0+,\xi)}{r^{\alpha +1}}d r\\
    &\eqqcolon H_1+H_2+H_3.
\end{align}
Since $\int_1^\infty 1/r^{\alpha}d r< \infty$, where $\alpha\in(3/2,2)$,
we obtain $H_1\to i \inner*{m}{z}$ as $\alpha\uparrow 2$.
From \eqref{eq:gamma}, we obtain
\begin{align}
    H_2
    = h(\alpha)\Gamma(-\alpha)\int_S|\inner*{z}{\xi}|^\alpha e^{-i\frac{\pi\alpha}{2}\mathrm{sgn}\inner*{z}{\xi}}q(0+,\xi)\rho(d\xi).
\end{align}
Using the limit $(2-\alpha)\Gamma(-\alpha) \to 1/2$ as $\alpha \uparrow 2$ and the definition of $\Sigma$, we have
\begin{align}
    H_2
    \to-\frac{1}{2}\int_S|\inner*{z}{\xi}|^2 q(0+,\xi)\rho(d\xi)=-\frac{1}{2}\inner*{z}{\Sigma z}.
\end{align}
Next we show the convergence of $H_3$.
When $1\le r<\infty$,
\begin{align}
    \abs{e^{i\inner{z}{r\xi}}-1}\cdot\frac{\abs*{q(r,\xi)-q(0+,\xi)}}{r^{\alpha+1}}
    \le 2\cdot 2M\cdot \frac{1}{r^2}
\end{align}
holds, and when $0< r<1$, 
\begin{align}
    \abs{e^{i\inner{z}{r\xi}}-1-i\inner{z}{r\xi}}\cdot\frac{\abs*{q(r,\xi)-q(0+,\xi)}}{r^{\alpha+1}}
    &\le \frac{|z|^2}{2}\cdot \frac{\abs*{q(r,\xi)-q(0+,\xi)}}{r}
\end{align}
holds.
By condition (G-3), we conclude that $H_3\to 0$ as $\alpha \uparrow 2$.
Combining the limits of $H_1, H_2$, and $H_3$, we obtain
\begin{equation}
    \Exp{}{e^{i \inner{z}{X^{\alpha}_1}}}\to \exp\cbra*{i\inner{m}{z}-\frac{1}{2}\inner*{z}{\Sigma z}}.
\end{equation}

We now show that $Y^\alpha_1$ converges in distribution to a stable distribution with stable index $1$.
For $U\in\mathbb{R}^{d\times d}$, let $\psi_\alpha(U)\coloneqq\Exp{}{e^{i \inner{U}{Y^\alpha_1}}}$.
By Lemma \ref{le:jointdist}, we have
\begin{align}
    &\Exp{}{e^{i \inner{U}{Y^\alpha_1}}}\\
    &=e^{-i \inner{U/c(\alpha)}{D(\alpha)}}\cdot \Exp{}{e^{i \inner{U/c(\alpha)}{[X^\alpha]_1}}}\\
    &=\exp\cbra*{\frac{h(\alpha)}{2c(\alpha)^{\alpha/2}}\!\int_S\rho(d\xi)\int_0^\infty\rbra*{e^{i r\inner*{\xi}{U\xi}}-1-i r\inner*{\xi}{U\xi}\bm{1}_{(0,1)}(r)}\frac{q(r^{1/2}c(\alpha)^{1/2},\xi)}{r^{1+\alpha/2}}d r}\label{eq:chf_g2}.
\end{align}
By condition (G-2) and \cite[Lemma 14.11]{ken-iti1999levy},
and applying Lebesgue's dominated convergence theorem, we obtain
\begin{equation}
    \psi_\alpha(U)\to\exp\cbra*{\int_S\rbra*{-\frac{\pi}{2}\abs*{\inner*{\xi}{U\xi}}-i \inner*{\xi}{U\xi}\log\abs*{\inner*{\xi}{U\xi}}+i \gamma \inner*{\xi}{U\xi}}q(0+,\xi)\rho(d\xi)}
\end{equation}
as $\alpha \uparrow 2$.

It remains to show that the two limit distributions are independent.
We have
\begin{align}
    &e^{i \rbra{\inner{z}{x}+\inner{x}{Ux}/c(\alpha)}}-1-i \inner{z}{x}\bm{1}_{\{|x|<1\}}\\
    &\hspace{3mm}=(e^{i \inner{z}{x}}-1)(e^{i \inner{x}{Ux}/c(\alpha)}-1)
    +(e^{i \inner{z}{x}}-1-i \inner{z}{x}\bm{1}_{\{|x|<1\}})+(e^{i \inner{x}{Ux}/c(\alpha)}-1)
\end{align}
for every $z\in\mathbb{R}^d$ and $U\in\mathbb{R}^{d\times d}$.
Therefore, by Lemma~\ref{le:jointdist}, we obtain
\begin{align}
    &\phantom{=}\Exp*{}{e^{i \rbra*{\inner{z}{X^\alpha_1}+\inner{U}{Y^\alpha_1}}}}\\
    &=\exp\cbra*{\int_{\mathbb{R}^d}(e^{i \inner{z}{x}}-1)(e^{i \inner{x}{Ux}/c(\alpha)}-1)\nu_\alpha(d x)}
    \cdot\Exp*{}{e^{i \inner{z}{X^\alpha_1}}}\cdot\Exp*{}{e^{i \inner{U}{Y^\alpha_1}}}.
\end{align}
To establish the independence, it suffices to show that
\begin{equation}
    \exp\cbra*{\int_{\mathbb{R}^d}(e^{i \inner{z}{x}}-1)(e^{i \inner{x}{Ux}/c(\alpha)}-1)\nu_\alpha(d x)}\to 1
\end{equation}
where $\alpha\uparrow 2$, i.e. the exponent converges to $0$ as $\alpha\uparrow 2$.
To show the convergence, we compute
\begin{align}
    &\phantom{=}\int_{\mathbb{R}^d}(e^{i \inner{z}{x}}-1)(e^{i \inner{x}{Ux}/c(\alpha)}-1)\nu_\alpha(d x)\label{eq:domi_g2}\\
    &=h(\alpha)\int_S\rho(d\xi)\int_0^\infty(e^{i r\inner{z}{\xi}}-1)(e^{i r^2\inner{\xi}{U\xi}/c(\alpha)}-1)\frac{q(r,\xi)}{r^{\alpha +1}}d r\\
    &=\frac{h(\alpha)}{c(\alpha)^{\alpha/2}}\int_S\rho(d\xi)\underbrace{\int_0^\infty(e^{i sc(\alpha)^{1/2}\inner{z}{\xi}}-1)(e^{i s^2\inner{\xi}{U\xi}}-1)\frac{q(sc(\alpha)^{1/2},\xi)}{s^{\alpha +1}}d s}_{(\star)}.
\end{align}
Thus, it remains to prove that the term $(\star)$ converges to $0$.
Using condition (G-2) and the definition of $c(\alpha)$, 
we find the integrand converges to $0$.
On the interval $(0,1)$, we apply condition (G-1) and the boundedness of the integrand to obtain
\begin{align}
    \int_0^1\abs{e^{i sc(\alpha)^{1/2}\inner{z}{\xi}}-1}\abs{e^{i s^2\inner{\xi}{U\xi}}-1}\cdot\frac{q(sc(\alpha)^{1/2},\xi)}{s^{\alpha +1}}d s
    &\le |z|\norm*{U}M\int_0^1 s^{2-\alpha}d s
    <\infty.
\end{align}
On the interval $[1,\infty)$, we again use condition (G-1), and obtain
\begin{align}
    \int_1^\infty\abs{e^{i sc(\alpha)^{1/2}\inner{z}{\xi}}-1}\abs{e^{i s^2\inner{\xi}{U\xi}}-1}\frac{q(sc(\alpha)^{\frac{1}{2}},\xi)}{s^{\alpha +1}}d s
    &\le\int_1^\infty 4\cdot\frac{M}{s^{\alpha +1}}d s
    <\infty.
\end{align}
Applying Lebesgue's dominated convergence theorem,
we conclude the proof of the independence of the limit distributions.
\end{proof}

\section{Asymptotics of implied volatility and option prices}

In this section we specialize to a one-dimensional exponential L\'evy model and derive first-order expansions of option prices and the ATM implied volatility skew as $\alpha \uparrow 2$.
The main inputs are the Carr--Madan pricing formula, the decay of the characteristic function, and an expansion of the L\'evy exponent.

\subsection{A perturbation model with stable-type jumps}

We consider the behavior of implied volatility in a pure-jump model driven by a scaled GTS process as the stable index \( \alpha \uparrow 2 \).  
According to Theorem~\ref{th:main_infinite},
this process converges to Brownian motion under suitable scaling,
allowing us to define a pure-jump exponential L\'evy model that approaches the Black--Scholes model.

From this point on, we set $d=1$, so that $S=\{-1,1\}$ and
\begin{equation}
\int_S f(\xi)\rho(d\xi)=\rho(\{1\})f(1)+\rho(\{-1\})f(-1).
\end{equation}

For  $\delta_0>0$ and  an $\mathbb{R}$-valued GTS process $X$ with $(m,\alpha,\rho,q)$,
we impose the following conditions on the model parameters:
\begin{enumerate}
    \item[(D-1)] \;$q(r,1)=O(e^{-\delta_0 r})$ as $r\to \infty$;
    \item[(D-2)] \;$\rho(\left\{-1\right\})>0$ and $q(0+,-1)>0$.
\end{enumerate}

\begin{definition}[Scaled pure-jump exponential L\'evy model]\label{de:scaledmodel}
    Let $\alpha\in(3/2,2)$ and $X^\alpha$ be a GTS process with parameters $(m_\alpha,\alpha,\rho,h(\alpha)q)$ satisfying the following conditions:
    \begin{itemize}
        \item $h(\alpha)=2-\alpha$, $X^\alpha_0=0$,  $q(0+,1)=0$,
        and 
        \begin{equation}
            m_\alpha =  -\int_{\mathbb{R}}(e^x-1-x\bm{1}_{\left\{|x|< 1\right\}})\nu_\alpha(d x),
        \end{equation}
        where $\nu_\alpha$ is the L\'evy measure of $X^\alpha$;
        \item the function $q$ satisfies (G-1), (G-2), and (G-3);
        \item there exists $\delta_0>1$ such that
        (D-1) is satisfied;
        \item (D-2) holds.
    \end{itemize}
    The \emph{scaled pure-jump exponential L\'evy model} is defined by  $S^\alpha\coloneqq \exp(X^\alpha)$.
\end{definition}

\begin{remark}
    Positive jumps are not excluded in this definition,
    since we do not assume $\rho(\left\{1\right\})=0$ but $q(0+,1)=0$.
    The finite-moment log-stable (FMLS) process provided in \cite{carr2003finite}
    does not have positive jumps, so this model is a generalization of the FMLS process.
\end{remark}

\begin{remark}
    From \cite[Proposition 8.20]{cont2004financial}, the scaled pure-jump exponential L\'evy model $e^{X^\alpha}$ is a $\mathbb{P}$-martingale.
\end{remark}

Using the properties of the GTS processes established in the previous section,  
we analyze the convergence of the model and the asymptotic behavior of its characteristic function.  
The following corollary follows straightforwardly from Theorem~\ref{th:main_infinite}.

\begin{corollary}[Convergence of the scaled pure-jump exponential L\'evy model]\label{co:explevytobm}
    Let $\alpha\in(3/2,2)$ and $X^\alpha$ be the scaled pure-jump exponential L\'evy model given in Definition \ref{de:scaledmodel}.
    Let
    \begin{align}
        \Sigma\coloneqq \int_{S}\xi\xi^{\top}q(0+,\xi)\rho(d\xi)=q(0+,-1)\rho(\left\{-1\right\}),
        \quad m\coloneqq -\frac{\Sigma}{2},
    \end{align}
    and $X$ be a Brownian motion with $X_1\sim N(m,\Sigma)$.
    Then $X^\alpha\xrightarrow{\mathrm{d}} X$ on the Skorokhod space $D(\mathbb{R})$ as $\alpha\uparrow 2$. 
\end{corollary}

\begin{remark}
    As a consequence, the limiting model $e^X$ inherits the martingale property under $\mathbb{P}$.
\end{remark}

The characteristic function of a stable process
decays as $e^{-|z|^\alpha}$ for large $|z|$.
We show that a similar decay property holds for the characteristic function of a GTS process.
We substitute a complex number $z - i\delta$ with $z,\delta\in\mathbb{R}$ into the characteristic function $g_t$ of the GTS process $X_t$ and study its decay rate,  
since we intend to use the Carr--Madan formula (see \cite{carr1999option}) to apply the main result of this paper,  
and this formula contains a term of the form $g_t(z - i\delta)$.
If we are interested only in the decay rate of \( g_t(z) \),  
a simpler proof is possible without assuming condition (D-1).
For completeness, we provide the proof of the following lemma in Appendix B.

\begin{lemma}\label{co:decayspeed}
    Let $\alpha\in(3/2,2)$ and $X^\alpha$ be the scaled pure-jump exponential L\'evy model defined in Definition \ref{de:scaledmodel}.
    Let $g^\alpha_t$ denote the characteristic function of $X^\alpha_t$, and take any $\delta\in[0,\delta_0]$.
    Then there exist $\gamma'>0$, $C'\in\mathbb{R}$, $z_0>0$, and $\alpha_0 \in(3/2,2)$ such that, for any $\alpha\in(\alpha_0,2)$ and $z\in\mathbb{R}$ with $|z|\ge z_0$,
    \begin{align}
        |g^\alpha_t(z-i\delta)|\le \exp(-t\gamma' |z|^\alpha+ tC')
    \end{align}
    holds.
\end{lemma}

Let $\mathbb{P}$ be the risk-neutral measure for both the scaled pure-jump exponential L\'evy model $S^\alpha = e^{X^\alpha}$  
and the Black--Scholes model $S = e^X$, which is the limit of $S^\alpha$ as $\alpha \uparrow 2$.
The Black--Scholes call option pricing formula is written as
\begin{align}
    C_{\mathrm{BS}}(t,K,\sigma)\coloneqq \Phi\rbra*{\frac{\log(1/K)+\sigma^2t/2}{\sigma\sqrt{t}}}-K\cdot\Phi\rbra*{\frac{\log(1/K)-\sigma^2t/2}{\sigma\sqrt{t}}},
\end{align}
where $K$ is a strike price, $t$ is the time to maturity, $\sigma$ is the volatility of the Black--Scholes model,
and $\Phi$ is the distribution function of the standard normal distribution.
We assume that the spot price $S_0=1$.
Since the Black--Scholes pricing function $\sigma \mapsto C_{\mathrm{BS}}(t,K,\sigma)$ is strictly increasing in $\sigma$,  
for any call price $C\in((1-K)^+,1)$, there exists a unique $\hat{\sigma} > 0$ such that $C = C_{\mathrm{BS}}(t, K, \hat{\sigma})$.  
This value $\hat{\sigma}$ is referred to as the \emph{implied volatility}.

\begin{definition}[The implied volatility of the scaled pure-jump exponential L\'evy model]
    Let $\alpha\in(3/2,2)$, $X^\alpha$ be the scaled pure-jump exponential L\'evy model given in Definition \ref{de:scaledmodel}, and
    the call price be given by $C(t,K,\alpha)\coloneqq \Exp{}{(e^{X_t^\alpha}-K)_+}$.
    For any $t,K>0$, let $\hat{\sigma}(t,K,\alpha)$ denote a unique positive real number satisfying $C(t,K,\alpha)= C_{\mathrm{BS}}(t,K,\hat{\sigma}(t,K,\alpha))$.
\end{definition}

Using Corollary~\ref{co:explevytobm} and the monotonicity of the Black--Scholes pricing formula,  
we obtain the following convergence results:
\begin{align}
    C(t,K,\alpha)&\to C_{\mathrm{BS}}(t,K,\sqrt{\Sigma}),\\
    \Prob{\mathbb{P}}{e^{X^\alpha_t}\ge K}&\to \Prob{\mathbb{P}}{e^{X_t}\ge K}=1-\Phi\rbra*{\frac{\log K+\Sigma t/2}{\sqrt{\Sigma t}}},\\
    \hat{\sigma}(t,K,\alpha)&\to \sqrt{\Sigma}.
\end{align}
We aim to compute the asymptotic expansion of the ATM implied volatility skew as the model converges to the Black--Scholes model.  
The \emph{implied volatility skew} refers to the slope of the implied volatility when considered as a function of strike.
The \emph{ATM volatility skew} is defined as the value of the implied volatility skew at $K = S_0 = 1$.
Using the absolute continuity of the law of the asset price $S_t$ and the implicit function theorem,  
we obtain the following well-known formula for the implied volatility skew.
Let $\phi$ denote the standard normal density and $\Phi$ its cumulative distribution function.  
Then the implied volatility skew is given by
\begin{equation}
    \partial_K\hat{\sigma}(t,K,\alpha)
    =\frac{\partial_K C(t,K,\alpha)-\partial_K C_{\mathrm{BS}}(t,K,\hat{\sigma}(t,K,\alpha))}{\partial_\sigma C_{\mathrm{BS}}(t,K,\hat{\sigma}(t,K,\alpha))},
\end{equation}
and the ATM implied volatility skew formula is
\begin{equation}
    \partial_K\hat{\sigma}(t,K,\alpha)|_{K=1}=\frac{\Phi(-\hat{\sigma}(t,1,\alpha)\sqrt{t}/2)-\Prob*{\mathbb{P}}{S^\alpha_t\ge 1}}{\sqrt{t}\phi(\hat{\sigma}(t,1,\alpha)\sqrt{t}/2)}.
\end{equation}

\subsection{Results: asymptotics of the implied volatility and option prices}

\begin{theorem}[Asymptotics of the prices]\label{le:asympofprice}
    Let $S^\alpha =\exp(X^\alpha)$ be the scaled pure-jump exponential L\'evy model defined in Definition \ref{de:scaledmodel}.
    Then, as $\alpha\uparrow 2$,
    the prices of the call option and the digital option have the following asymptotics:
    \begin{align}
        C(t,K,\alpha)-C_{\mathrm{BS}}(t,K,\sqrt{\Sigma})&=(2-\alpha)A_{\mathrm{call}}(t,K)+o(2-\alpha),\\
        \Prob{\mathbb{P}}{e^{X_t^\alpha}\ge K}-\Prob{\mathbb{P}}{e^{X_t}\ge K}
        &=(2-\alpha)A_{\mathrm{dig}}(t,K)+o(2-\alpha),
    \end{align}
    where
    \begin{align}
        A_{\mathrm{call}}(t,K)&\coloneqq \frac{1}{2\pi}\int_{-\infty}^{\infty}e^{-i z\log K}
        \cdot\frac{g_t(z-i)\cdot tC(z,1)}{-z^2+i z}d z,\\
        A_{\mathrm{dig}}(t,K)&\coloneqq \frac{1}{\pi}\int_{0}^{\infty}\Re\sbra*{e^{-i z\log K}\cdot\frac{g_t(z)\cdot tC(z,0)}{i z}}d z,\\
        C(z,\delta)&\coloneqq \Sigma\cdot\cbra*{-\frac{1}{2}\cdot(\delta+i z)^2 \log(\delta+i z)+\frac{3-2\gamma_E}{4}\cdot(\delta+i z)^2-\frac{3-2\gamma_E}{4}\cdot(\delta+i z)}\\
        &\hspace{5mm} + \int_S\rho(d\xi)\int_0^\infty(e^{(\delta+i z)r\xi}-1-(\delta+i z)r\xi\bm{1}_{(0,1)}(r))\cdot\frac{q(r,\xi)-q(0+,\xi)}{r^{3}}d r\\
        &\hspace{5mm} -\cbra*{\int_S\rho(d\xi)\int_0^\infty(e^{r\xi}-1-r\xi\bm{1}_{(0,1)}(r))\cdot\frac{q(r,\xi)-q(0+,\xi)}{r^{3}}d r}\cdot(\delta+i z),
    \end{align}
    $g_t$ is the characteristic function of $X_t$, the log-price under the Black--Scholes model, and $\gamma_E$ is Euler's constant.
\end{theorem}

As a consequence of Theorem \ref{le:asympofprice}, we obtain the asymptotic expansion of the ATM implied volatility skew, which is stated in the following corollary. 
The proofs of Theorem \ref{le:asympofprice} and Corollary \ref{th:main} are provided in Subsections 4.4 and 4.5, respectively.

\begin{corollary}\label{th:main}
    Let $S^\alpha =\exp(X^\alpha)$ be the scaled pure-jump exponential L\'evy model defined in Definition \ref{de:scaledmodel}.
    Then the ATM implied volatility skew of this model has the asymptotic expansion
    \begin{align}
        \partial_K\hat{\sigma}(t,K,\alpha)|_{K=1}=(2-\alpha)A_{\mathrm{ATM}}(t)+o(2-\alpha),\quad \alpha\uparrow 2,
    \end{align}
    where
    \begin{align}
        A_{\mathrm{ATM}}(t)&\coloneqq
        -\frac{e^{\Sigma t/8}}{\sqrt{t}}\left(\sqrt{\frac{\pi}{2}}\,A_{\mathrm{call}}(t,1)+\sqrt{2\pi}\,A_{\mathrm{dig}}(t,1)\right).
    \end{align}
\end{corollary}

\begin{remark}
    A further analysis of the maturity dependence of $A_{\mathrm{ATM}}(t)$, together with a numerical illustration of Corollary~\ref{th:main}, is given in Section~5.
\end{remark}

\subsection{Asymptotic expansion of the cumulant}

Before we prove the asymptotics of the prices,
it is necessary to obtain the asymptotic expansion of the characteristic function of the log-price $X_t^\alpha$.
We define the difference between the cumulant generating functions of $X^\alpha_1$ and $X_1$ as
\begin{align}\label{eq:defoflambda}
    \Lambda_\alpha(z)
    &\coloneqq i m_\alpha z+\int_{\mathbb{R}}(e^{i zu}-1-i zu\bm{1}_{\left\{|u|<1\right\}})\nu_\alpha(du)-\cbra*{i mz-\frac{\Sigma}{2}z^2}.
\end{align}

To apply the Carr--Madan formula, which involves the term $g_t^\alpha(z - i \delta)$,  
we derive the asymptotic expansion of $\Lambda_\alpha(z - i \delta)$.

\begin{lemma}[The asymptotic expansion of the cumulant]\label{le:asympofcumulant}
    Let $S^\alpha =\exp(X^\alpha)$ be the scaled pure-jump exponential L\'evy model defined in Definition \ref{de:scaledmodel}
    and $\Lambda_\alpha$ be defined by \eqref{eq:defoflambda}.
    For any $\delta\in[0,\delta_0]$, there exist functions $C(z,\delta)$ and $D(z,\alpha)$ satisfying the following properties:
    \begin{itemize}
        \item $|C(z,\delta)|$ is dominated by a third-order polynomial with respect to $|z|$;
        \item there exists $\alpha_1\in (3/2,2)$ such that $|D(z,\alpha)/(2-\alpha)|$ is bounded by a third-order polynomial with respect to $|z|$ that does not depend on $\alpha\in(\alpha_1,2)$;
        \item for each $z\in\mathbb{R}$, $D(z,\alpha)=o(2-\alpha)$ as $\alpha\uparrow 2$.
    \end{itemize}
    The asymptotic expansion of $\Lambda_\alpha(z - i \delta)$ is given by
    \begin{align}
        \Lambda_\alpha(z-i\delta)=(2-\alpha)C(z,\delta)+D(z,\alpha),\qquad \alpha\uparrow 2.
    \end{align}
\end{lemma}

\begin{remark}
    The constant $\delta_0 > 0$ is specified in condition (D-1).
\end{remark}

\begin{proof}[Proof of Lemma \ref{le:asympofcumulant}]
From $q(0+,1)=0$, we decompose $\Lambda_\alpha(z-i\delta)$ as
\begin{align}
    &(2-\alpha)\rho(\left\{-1\right\})\int_0^\infty(e^{-i zr-\delta r}-1-(-i zr-\delta r))\cdot\frac{q(0+,-1)}{r^{\alpha+1}}d r+\frac{\Sigma}{2}(z-i \delta)^2 \\
    &\hspace{5mm}- (2-\alpha)\rho(\left\{-1\right\})\int_0^\infty(i zr+\delta r)\bm{1}_{[1,\infty)}(r)\cdot\frac{q(0+,-1)}{r^{\alpha+1}}d r \\
    &\hspace{5mm}+ (2-\alpha)\rho(\left\{-1\right\})\int_0^\infty(e^{-i zr-\delta r}-1-(-i zr-\delta r)\bm{1}_{(0,1)}(r))\cdot\frac{q(r,-1)-q(0+,-1)}{r^{\alpha+1}}d r \\
    &\hspace{5mm}+ (2-\alpha)\rho(\left\{1\right\})\int_0^\infty(e^{i zr+\delta r}-1-(i zr+\delta r)\bm{1}_{(0,1)}(r))\cdot\frac{q(r,1)}{r^{\alpha+1}}d r \\
    &\hspace{5mm}+i (z-i\delta)(m_\alpha-m)\\
    &\eqqcolon J_1+J_2+J_3+J_4+J_5.
\end{align}
We first derive the asymptotic expansion of $J_1$.
In the same manner as in the proof of Theorem 14.10 and Lemma 14.11 in \cite{ken-iti1999levy},
we obtain
\begin{align}
    \int_0^\infty (e^{wr}-1-wr)\cdot\frac{1}{r^{\alpha +1}}d r=\Gamma(-\alpha)(-w)^\alpha,%
\end{align}
where $w\in\Set*{w\in \mathbb{C}}{\Re w\le 0,\; w\neq 0}$ and $\log(-w)$ is taken to be the principal value ($-\pi< \arg (-w)\le \pi$).
Therefore, taking $w=-\delta-i z \;(z\in\mathbb{R})$ and recalling $\Sigma =q(0+,-1)\rho(\left\{-1\right\})$,
we obtain
\begin{align}
    J_1=\Sigma\cdot (2-\alpha)\Gamma(-\alpha)\cdot (\delta+i z)^\alpha - \Sigma\cdot \frac{1}{2}\cdot (\delta+i z)^2.
\end{align}
Expanding the maps $\alpha\mapsto (2-\alpha)\Gamma(-\alpha)$ and $\alpha\mapsto (\delta+i z)^\alpha$ around $\alpha=2$, we have
\begin{align}
    (2-\alpha)\Gamma(-\alpha)
    &= \frac{1}{2}+\frac{3-2\gamma_E}{4}(2-\alpha)+o(2-\alpha),\\
    (\delta+i z)^\alpha
    &= (\delta+i z)^2-(2-\alpha)(\delta+i z)^2\log(\delta+i z)+o(2-\alpha),
\end{align}
where $\gamma_E$ is Euler's constant and $\log(\delta+i z)$ is taken to be the principal value ($-\pi< \arg (\delta+i z)\le \pi$).
We define
\[
    C_1(z,\delta)\coloneqq \Sigma\cdot\left\{-\frac{1}{2}\cdot(\delta+i z)^2 \log(\delta+i z)+\frac{3-2\gamma_E}{4}\cdot(\delta+i z)^2\right\}.
\]
Then we obtain the following expansion:
\[
    J_1=(2-\alpha)C_1(z,\delta)+o(2-\alpha).
\]
Let $D_1(z,\alpha)\coloneqq J_1-(2-\alpha)C_1(z,\delta)$. Then
$|C_1(z,\delta)|$ is bounded by a third-order polynomial with respect to $|z|$, and there exists $\alpha_1\in(3/2,2)$ such that $|D_1(z,\alpha)/(2-\alpha)|$ is bounded by a third-order polynomial with respect to $|z|$ that does not depend on $\alpha\in(\alpha_1,2)$.

Next, we treat $J_2,\dots,J_5$.

For $J_2$, we set
\begin{align}
    C_2(z,\delta)
    &\coloneqq
    -\rho(\{-1\})\int_1^\infty (iz+\delta)\frac{q(0+,-1)}{r^2}dr
    =-\Sigma(\delta+iz).
\end{align}
Then $J_2/(2-\alpha)\to C_2(z,\delta)$ as $\alpha\uparrow 2$, and
\begin{align}
    \abs*{\frac{J_2-(2-\alpha)C_2(z,\delta)}{2-\alpha}}
    &\le
    \Sigma |\delta+iz| \int_1^\infty \abs*{\frac{1}{r^\alpha}-\frac{1}{r^2}}dr
    \le
    2\Sigma (\delta+|z|),
\end{align}
so both $C_2(z,\delta)$ and $D_2(z,\alpha)\coloneqq J_2-(2-\alpha)C_2(z,\delta)$ satisfy the required bounds.

For $J_3$ and $J_4$, we write
\begin{align}
    R_-(r)
    &\coloneqq
    e^{-izr-\delta r}-1-(-izr-\delta r)\bm{1}_{(0,1)}(r),\\
    R_+(r)
    &\coloneqq
    e^{izr+\delta r}-1-(izr+\delta r)\bm{1}_{(0,1)}(r).
\end{align}
There exists a constant $c_\delta>0$ such that, for all $r>0$,
\begin{align}
    |R_-(r)|
    &\le
    c_\delta \rbra*{(1+|z|^2)r^2\bm{1}_{(0,1)}(r)+(1+|z|)r\bm{1}_{[1,\infty)}(r)},\\
    |R_+(r)|
    &\le
    c_\delta \rbra*{(1+|z|^2)r^2\bm{1}_{(0,1)}(r)+(1+|z|)e^{\delta r}r\bm{1}_{[1,\infty)}(r)}.
\end{align}
We define
\begin{align}
    C_3(z,\delta)
    &\coloneqq
    \rho(\{-1\})\int_0^\infty R_-(r)\frac{q(r,-1)-q(0+,-1)}{r^3}dr,\\
    C_4(z,\delta)
    &\coloneqq
    \rho(\{1\})\int_0^\infty R_+(r)\frac{q(r,1)}{r^3}dr.
\end{align}
These integrals are absolutely convergent. Indeed, on $(0,1)$ the above bounds give an integrand of order
$(1+|z|^2)|q(r,\xi)-q(0+,\xi)|/r$ for $J_3$ and
$(1+|z|^2)q(r,1)/r$ for $J_4$, which are integrable by (G-3) and $q(0+,1)=0$.
On $[1,\infty)$, boundedness of $q$ yields integrability for $J_3$, while for $J_4$ condition (D-1) implies that
$e^{\delta r}q(r,1)$ is bounded for $\delta\in[0,\delta_0]$.
Therefore, by dominated convergence,
\begin{align}
    \frac{J_3}{2-\alpha}\to C_3(z,\delta),
    \qquad
    \frac{J_4}{2-\alpha}\to C_4(z,\delta)
\end{align}
as $\alpha\uparrow 2$.
Now let
\begin{align}
    D_3(z,\alpha)
    &\coloneqq
    J_3-(2-\alpha)C_3(z,\delta),\\
    D_4(z,\alpha)
    &\coloneqq
    J_4-(2-\alpha)C_4(z,\delta).
\end{align}
For $\alpha\in(\alpha_1,2)$ with $\alpha_1>3/2$,
\begin{align}
    \abs*{\frac{1}{r^{\alpha+1}}-\frac{1}{r^3}}
    \le
    2r^{-3}\bm{1}_{(0,1)}(r)+2r^{-5/2}\bm{1}_{[1,\infty)}(r).
\end{align}
Hence
\begin{align}
    \abs*{\frac{D_3(z,\alpha)}{2-\alpha}}
    &\le
    c_\delta (1+|z|^2)\int_0^1 \frac{|q(r,-1)-q(0+,-1)|}{r}dr
    +
    c_\delta (1+|z|)\int_1^\infty \frac{dr}{r^{3/2}}
\end{align}
and
\begin{align}
    \abs*{\frac{D_4(z,\alpha)}{2-\alpha}}
    &\le
    c_\delta (1+|z|^2)\int_0^1 \frac{q(r,1)}{r}dr
    +
    c_\delta (1+|z|)\int_1^\infty \frac{e^{\delta r}q(r,1)}{r^{3/2}}dr.
\end{align}
All integrals on the right-hand sides are finite by (G-3), $q(0+,1)=0$, and (D-1).
Thus $C_3(z,\delta)$ and $C_4(z,\delta)$ are bounded by second-order polynomials in $|z|$, and
$D_3(z,\alpha)/(2-\alpha)$ and $D_4(z,\alpha)/(2-\alpha)$ satisfy the same bound uniformly in $\alpha\in(\alpha_1,2)$.

Finally, for $J_5=i(z-i\delta)(m_\alpha-m)$, the definition of $m_\alpha$, together with the argument for $J_1$, shows that
\begin{align}
    m_\alpha-m
    =
    (2-\alpha)K+o(2-\alpha),
\end{align}
where
\begin{align}
    K
    \coloneqq
    \Sigma\frac{2\gamma_E+1}{4}
    -
    \int_S \rho(d\xi)\int_0^\infty
    \rbra*{e^{r\xi}-1-r\xi\bm{1}_{(0,1)}(r)}
    \frac{q(r,\xi)-q(0+,\xi)}{r^3}dr.
\end{align}
The integral defining $K$ is finite by the same estimates as above:
on $(0,1)$ the kernel is $O(r^2)$, and on $[1,\infty)$ we use boundedness of $q$ for $\xi=-1$ and (D-1) for $\xi=1$.
Therefore, if we set
\begin{align}
    C_5(z,\delta)
    \coloneqq
    (\delta+iz)K,
    \qquad
    D_5(z,\alpha)
    \coloneqq
    i(z-i\delta)(m_\alpha-m)-(2-\alpha)C_5(z,\delta),
\end{align}
then $C_5(z,\delta)$ is bounded by a first-order polynomial in $|z|$, and
$D_5(z,\alpha)/(2-\alpha)$ is uniformly bounded by the same order.

Summarizing the above arguments, when we define
\begin{align}
    C(z,\delta)&\coloneqq C_1(z,\delta)+C_2(z,\delta)+C_3(z,\delta)+C_4(z,\delta)+C_5(z,\delta),\\
    D(z,\alpha) &\coloneqq D_1(z,\alpha)+D_2(z,\alpha)+D_3(z,\alpha)+D_4(z,\alpha)+D_5(z,\alpha),
\end{align}
we can obtain
\begin{align}
    \Lambda_\alpha(z-i \delta)=(2-\alpha)C(z,\delta)+D(z,\alpha)
\end{align}
and $D(z,\alpha)=o(2-\alpha)$.
Moreover, $|C(z,\delta)|$ is bounded by a third-order polynomial with respect to $|z|$,
and if $\alpha\in (\alpha_1,2)$, $|D(z,\alpha)/(2-\alpha)|$ is also bounded by a third-order polynomial with respect to $|z|$.
\end{proof}

\subsection{Proof: asymptotic expansions of call and digital option prices}

\begin{proof}[Proof of Theorem \ref{le:asympofprice}]
Let $\beta\in(0,\delta_0-1]$ and $M_\alpha\coloneqq (2-\alpha)^{-1/3}$.
Using the Carr--Madan formula \cite{carr1999option}, we obtain
\begin{align}
    &C(t,K,\alpha)-C_{\mathrm{BS}}(t,K,\sqrt{\Sigma})\\
    =\;&\frac{e^{-\beta\log K}}{2\pi}\int_{-\infty}^{\infty}e^{-i z\log K}
    \cdot\frac{g^\alpha_t(z-i(\beta+1))-g_t(z-i(\beta+1))}{\beta^2+\beta-z^2+i(2\beta+1)z}d z \\
    =\;& \frac{e^{-\beta\log K}}{2\pi}\int_{|z|> M_\alpha}e^{-i z\log K}
    \cdot\frac{g^\alpha_t(z-i(\beta+1))-g_t(z-i(\beta+1))}{\beta^2+\beta-z^2+i(2\beta+1)z}d z \\
    & \hspace{5mm} +\frac{e^{-\beta\log K}}{2\pi}\int_{|z|\le M_\alpha}e^{-i z\log K}
    \cdot\frac{g^\alpha_t(z-i(\beta+1))-g_t(z-i(\beta+1))}{\beta^2+\beta-z^2+i(2\beta+1)z}d z \\
    \eqqcolon \;& T_1+T_2.
\end{align}
We first consider the term, $T_1$.
We have $|\beta^2+\beta-z^2+i(2\beta+1)z|\ge \beta$.
By Lemma \ref{co:decayspeed},
if we take $\alpha\in(\alpha_0,2)$ such that $M_\alpha=(2-\alpha)^{-1/3}>z_0$,
we obtain
\begin{align}
    |T_1|
    &\le\frac{e^{-\beta\log K}}{2\pi}\int_{|z|> M_\alpha}\frac{|g^\alpha_t(z-i(\beta+1))-g_t(z-i(\beta+1))|}{|\beta^2+\beta-z^2+i(2\beta+1)z|}d z\\
    &\le\frac{e^{-\beta\log K}}{2\pi}\cdot\frac{1}{\beta}\cdot\int_{|z|> M_\alpha}|g^\alpha_t(z-i(\beta+1))-g_t(z-i(\beta+1))|d z\\
    &\le\frac{e^{-\beta\log K}}{2\pi\beta}\cdot 2e^{tC'}\cdot\int_{|z|> M_\alpha}e^{-t\gamma'|z|^\alpha}d z\\
    &\le\frac{e^{-\beta\log K+tC'}}{\pi\beta}\cdot C''\exp(-\tilde{\gamma}M_\alpha^\alpha),
\end{align}
where $\tilde{\gamma}\in (0,t\gamma')$ and $C''\coloneqq \sup_{\alpha\in(\alpha_0,2)}\int_{|z|>M_\alpha}e^{-(t\gamma' - \tilde{\gamma}) |z|^\alpha}dz$.
Hence, it follows that
\begin{align}
    \frac{|T_1|}{2-\alpha}=O\rbra*{\frac{\exp(-\tilde{\gamma}(2-\alpha)^{-\alpha/3})}{2-\alpha}}=o(2-\alpha).
\end{align}
Next, $T_2$ is expressed as
\begin{align}
    T_2
    &=\frac{e^{-\beta\log K}}{2\pi}\int_{|z|\le M_\alpha}e^{-i z\log K}
    \cdot\frac{g_t(z-i(\beta+1))\cdot\cbra*{e^{t\Lambda_\alpha(z-i(\beta+1))}-1}}{\beta^2+\beta-z^2+i(2\beta+1)z}d z\\
    &=\frac{e^{-\beta\log K}}{2\pi}\int_{|z|\le M_\alpha}\!\!\!\!e^{-i z\log K}\!
    \cdot\!\frac{g_t(z-i(\beta+1))\cdot\cbra*{e^{t\Lambda_\alpha(z-i(\beta+1))}-1-t\Lambda_\alpha(z-i(\beta+1))}}{\beta^2+\beta-z^2+i(2\beta+1)z}d z\label{eq:priceinner1}\\
    &\hspace{1cm}+\frac{e^{-\beta\log K}}{2\pi}\int_{|z|\le M_\alpha}e^{-i z\log K}
    \cdot\frac{g_t(z-i(\beta+1))\cdot t\Lambda_\alpha(z-i(\beta+1))}{\beta^2+\beta-z^2+i(2\beta+1)z}d z\label{eq:priceinner2}\\
    &\eqqcolon U_1+U_2.
\end{align}
We next analyze the term $U_1$. For simplicity, we introduce the notation
\begin{align}
    B_\alpha(z)\coloneqq \frac{\Lambda_\alpha(z-i(\beta+1))}{2-\alpha}
    =C(z,\beta+1)+\frac{D(z,\alpha)}{2-\alpha}.
\end{align}
By Lemma \ref{le:asympofcumulant}, for any $\alpha\in(\alpha_1,2)$,
$|B_\alpha(z)|$ is bounded by a third-order polynomial with respect to $|z|$ that does not depend on $\alpha$.
Therefore, there exists $N>0$ such that $|tB_\alpha(z)|\le N M_\alpha^3$ for all $|z|\le M_\alpha$.
From the Taylor expansion up to second order of the function $w\mapsto e^{(2-\alpha)w}$ around $w=0$,
\begin{align}
    |e^{w(2-\alpha)}-1-w(2-\alpha)|&=\abs*{\frac{1}{2\pi i}\int_{\{|\zeta|=2NM_\alpha^3\}}\frac{e^{\zeta(2-\alpha)}}{(\zeta-0)^2(\zeta-w)}d\zeta}|w-0|^2\\
    &\le \frac{e^{2NM_\alpha^3(2-\alpha)}}{2NM_\alpha^3(2NM_\alpha^3-|w|)}\cdot |w|^2
\end{align}
holds, where $|w|<2NM_\alpha^3$.
Recalling $|tB_\alpha(z)|\le NM_\alpha^3$ when $|z|\le M_\alpha$, we obtain
\begin{align}
    |e^{t\Lambda_\alpha(z-i(\beta+1))}-1-t\Lambda_\alpha(z-i(\beta+1))|
    &= |e^{tB_\alpha(z)(2-\alpha)}-1-tB_\alpha(z)(2-\alpha)|\\
    &\le \frac{e^{2NM_\alpha^3(2-\alpha)}\cdot |tB_\alpha(z)|^2}{2N^2M_\alpha^6}\\
    &\le \frac{e^{2N}t^2}{2 N^2}\cdot|B_\alpha(z)|^2\cdot(2-\alpha)^2.
\end{align}
Recall that $g_t$ is the characteristic function of the normal distribution,
$|B_\alpha(z)|$ is bounded by a sixth-order polynomial with respect to $|z|$ where $\alpha\in(\alpha_1,2)$, and $|\beta^2+\beta-z^2+i(2\beta+1)z|\ge \beta $.
Therefore, by Lebesgue's dominated convergence theorem, we conclude that $U_1/(2-\alpha)\to 0$.
Finally, we consider $U_2$. From the definition of $B_\alpha(z)$, we can write
\begin{align}
    \frac{U_2}{2-\alpha}
    &=\frac{e^{-\beta\log K}}{2\pi}\int_{|z|\le M_\alpha}e^{-i z\log K}
    \cdot\frac{g_t(z-i(\beta+1))\cdot tB_\alpha(z)}{\beta^2+\beta-z^2+i(2\beta+1)z}d z.
\end{align}
Recall that $B_\alpha(z)$ has the limit $C(z,\beta+1)$ and is bounded by a third-order polynomial with respect to $|z|$ where $\alpha\in(\alpha_1,2)$.
Therefore, by Lebesgue's dominated convergence theorem, we can define
\begin{align}
    A(t,K,\beta)
    &\coloneqq\lim_{\alpha\uparrow 2}\frac{U_2}{2-\alpha}
    =\frac{e^{-\beta\log K}}{2\pi}\int_{-\infty}^{\infty}e^{-i z\log K}
    \cdot\frac{g_t(z-i(\beta+1))\cdot tC(z,\beta+1)}{\beta^2+\beta-z^2+i(2\beta+1)z}d z.
\end{align}
Recalling that the price given by the Carr--Madan formula does not depend on $\beta$,
we find the function $\beta\mapsto A(t,K,\beta)$ is a constant on $(0,\delta_0-1]$.
Hence, we define
\begin{equation}
    A_{\mathrm{call}}(t,K)\coloneqq\lim_{\beta\downarrow 0} A(t,K,\beta)
    =\frac{1}{2\pi}\int_{-\infty}^{\infty}e^{-i z\log K}
    \cdot\frac{g_t(z-i)\cdot tC(z,1)}{-z^2+i z}d z,
\end{equation}
and the asymptotic expansion of the call price is given by
\begin{align}
    C(t,K,\alpha)-C_{\mathrm{BS}}(t,K,\sqrt{\Sigma})=(2-\alpha)A_{\mathrm{call}}(t,K)+o(2-\alpha),\quad \alpha\uparrow 2.
\end{align}

According to \cite[Theorem 4.3]{lee2004option}, the asymptotic expansion of the digital option price $\Prob{\mathbb{P}}{e^{X_t^\alpha}\ge K}$ can be derived by similar arguments as those used for the call option price.
Therefore, we obtain the desired result and complete the proof.

\end{proof}

\subsection{Proof: asymptotic expansion of the ATM volatility skew}

Finally, we derive the asymptotic expansion of the ATM volatility skew based on the asymptotic behavior of the option prices.

\begin{proof}[Proof of Corollary \ref{th:main}]
Let $\phi$ be the density of the standard normal distribution. We define
\begin{align}
    \mathrm{erf}(y)\coloneqq \frac{2}{\sqrt{\pi}}\int_0^ye^{-u^2}du,\quad
    F(x,\theta)\coloneqq \int_0^\theta\phi\rbra*{\frac{x}{v}+\frac{v}{2}}d v.
\end{align}
According to \cite[Lemma 3.1]{roper2009relationship}, we have
\begin{align}
    C_\mathrm{BS}(t,K,\sigma)
    &=(1-K)_+ + F(\log(1/K),\sigma\sqrt{t}).
\end{align}
It then follows that
\begin{align}
    C(t,1,\alpha)&=C_{\mathrm{BS}}(t,1,\hat{\sigma}(t,1,\alpha))=F(0,\hat{\sigma}(t,1,\alpha)\sqrt{t}),\\
    &C_{\mathrm{BS}}(t,1,\sqrt{\Sigma})=F(0,\sqrt{\Sigma t}).
\end{align}
By applying the Taylor expansion of the error function \(\mathrm{erf}(y)\) at \(y=\sqrt{\Sigma t}/(2\sqrt{2})\), and noting that \(F(0,\theta)=\mathrm{erf}(\theta/(2\sqrt{2}))\), we obtain
\begin{align}
    F(0,\hat{\sigma}(t,1,\alpha)\sqrt{t})-F(0,\sqrt{\Sigma t})
    = \frac{e^{-\Sigma t/8}}{\sqrt{2\pi}}(\hat{\sigma}(t,1,\alpha)\sqrt{t}-\sqrt{\Sigma t})+o(\hat{\sigma}(t,1,\alpha)\sqrt{t}-\sqrt{\Sigma t}).
\end{align}
Therefore, it follows that
\begin{align}
    &\hat{\sigma}(t,1,\alpha)\sqrt{t}-\sqrt{\Sigma t}\\
    &= \sqrt{2\pi}e^{\Sigma t/8}(F(0,\hat{\sigma}(t,1,\alpha)\sqrt{t})-F(0,\sqrt{\Sigma t}))+o(\hat{\sigma}(t,1,\alpha)\sqrt{t}-\sqrt{\Sigma t})\\
    &= (2-\alpha)A_{\mathrm{IV}}(t)+o(2-\alpha),
\end{align}
where
\begin{align}
    A_{\mathrm{IV}}(t)\coloneqq \sqrt{2\pi}e^{\Sigma t/8}\cdot A_{\mathrm{call}}(t,1).
\end{align}
Applying the Taylor expansion to the map $x\mapsto \Phi(-x)$ around $x=\sqrt{\Sigma t}/2$ and that of the map $x\mapsto 1/\phi(x)$ around $x=\sqrt{\Sigma t}/2$,
we obtain
\begin{align}
    \partial_K\hat{\sigma}(t,K,\alpha)|_{K=1}
    &=\frac{\Phi(-\hat{\sigma}(t,1,\alpha)\sqrt{t}/2)-\Prob*{\mathbb{P}}{S^\alpha_t\ge 1}}{\sqrt{t}\phi(\hat{\sigma}(t,1,\alpha)\sqrt{t}/2)} \\
    &=(2-\alpha)A_{\mathrm{ATM}}(t)+o(2-\alpha),
\end{align}
where
\begin{align}
    A_{\mathrm{ATM}}(t)&\coloneqq %
    -\frac{e^{\Sigma t/8}}{\sqrt{t}}\left(\sqrt{\frac{\pi}{2}}\,A_{\mathrm{call}}(t,1)+\sqrt{2\pi}\,A_{\mathrm{dig}}(t,1)\right).
\end{align}
\end{proof}

\section{Exact order and numerical experiment}

We next study the short-time and long-time behavior of \(A_{\mathrm{ATM}}(t)\).
Whenever it exists, we write
\[
q_r(0+,\xi)\coloneqq \lim_{r\downarrow 0}\frac{q(r,\xi)-q(0+,\xi)}{r}.
\]
For the analysis below, in addition to \((\mathrm{G}\text{-}1)\)--\((\mathrm{G}\text{-}3)\), we assume the following:
\begin{enumerate}
    \item[(G-4)] For \(\rho\)-a.e. \(\xi\in S\), \(b(\xi)\coloneqq q_r(0+,\xi)\) exists, \(\int_S |b(\xi)|\,\rho(d\xi)<\infty\), and
    \[
    \int_S \rho(d\xi)\int_0^1 |\Delta q(r,\xi)|\,r^{-2}dr<\infty,
    \]
    where \(\Delta q(r,\xi)\coloneqq q(r,\xi)-q(0+,\xi)-b(\xi)r\bm{1}_{(0,1)}(r)\).
\end{enumerate}
Assumption (G-4) is used throughout this section, both in the short-time analysis and in the analytic estimates for the long-time analysis.
It is satisfied, for example, by standard exponentially tempered specifications.
In particular, in the one-dimensional case it covers the usual tempered-stable/CGMY-type specification \(q(r,1)=e^{-Mr}\) and \(q(r,-1)=e^{-Gr}\), with the constants absorbed into \(\rho\).

\subsection{Short-time and long-time behavior of \texorpdfstring{\(A_{\mathrm{ATM}}(t)\)}{A\_ATM(t)}}

The analytic estimates used below are established in Appendix C.

By Corollary~4.2, at \(K=1\),
\begin{equation}
A_{\mathrm{ATM}}(t)
=
-\frac{e^{\Sigma t/8}}{\sqrt{t}}
\left(
\sqrt{\frac{\pi}{2}}\,A_{\mathrm{call}}(t,1)
+
\sqrt{2\pi}\,A_{\mathrm{dig}}(t,1)
\right),
\end{equation}
where
\[
A_{\mathrm{call}}(t,1)
=
\frac{1}{2\pi}\int_{\mathbb{R}} g_t(z-i)\frac{tC(z,1)}{-z^2+iz}\,dz,
\quad
A_{\mathrm{dig}}(t,1)
=
\frac{1}{\pi}\int_0^\infty \Re\!\left[g_t(z)\frac{tC(z,0)}{iz}\right]dz,
\]
and \(g_t(z)=\exp\!\left(-i\Sigma tz/2-\Sigma tz^2/2\right)\).
We first study the short-time behavior.

\begin{proposition}\label{pr:a_atm_short}
Under $(\mathrm{G}\text{-}4)$, as \(t\downarrow0\),
\begin{equation}
A_{\mathrm{ATM}}(t)
=
-\frac{\sqrt{2\pi}}{4}\frac{1}{\sqrt{t}}+O(|\log t|).
\end{equation}
\end{proposition}

\begin{proof}
Since \(m=-\Sigma/2\), we have \(|g_t(z-i)|=e^{-\Sigma t z^2/2}\). Together with Lemma~\ref{le:analytic_lemma_call},
\[
\left|\frac{C(z,1)}{-z^2+iz}\right|
\le c\bigl(1+\log(2+|z|)\bigr),
\qquad z\in\mathbb{R},
\]
this gives
\[
|A_{\mathrm{call}}(t,1)|
\le
ct\int_{\mathbb{R}}e^{-\Sigma t z^2/2}\bigl(1+\log(2+|z|)\bigr)\,dz.
\]
After \(u=\sqrt{t}\,z\),
\[
|A_{\mathrm{call}}(t,1)|
\le
c\sqrt{t}\int_{\mathbb{R}}e^{-\Sigma u^2/2}
\left(1+\log\left(2+\frac{|u|}{\sqrt{t}}\right)\right)\,du.
\]
Since \(\log(2+|u|/\sqrt{t})\le \frac12|\log t|+\log(2+|u|)\) for \(t\in(0,1]\), we obtain
\begin{equation}
A_{\mathrm{call}}(t,1)=O(\sqrt{t}\,|\log t|).
\end{equation}

Next define
\begin{align}
C_1(z)
&\coloneqq
\int_S \rho(d\xi)\, q_r(0+,\xi)
\left(
\int_0^1 \left(e^{izr\xi}-1-izr\xi\right)\frac{dr}{r^2}
-iz\int_0^1 \left(e^{r\xi}-1-r\xi\right)\frac{dr}{r^2}
\right),\\
C_2(z)
&\coloneqq
\int_S \rho(d\xi)\int_0^1
\left(e^{izr\xi}-1-izr\xi\right)\frac{\Delta q(r,\xi)}{r^3}\,dr\\
&\quad
+\int_S \rho(d\xi)\int_1^\infty
\left(e^{izr\xi}-1\right)\frac{\Delta q(r,\xi)}{r^3}\,dr\\
&\quad
-iz\int_S \rho(d\xi)\int_0^\infty
\left(e^{r\xi}-1-r\xi\bm{1}_{(0,1)}(r)\right)\frac{\Delta q(r,\xi)}{r^3}\,dr.
\end{align}
Put
\[
G(z,\delta)\coloneqq
-\frac12(\delta+iz)^2\log(\delta+iz)
+\frac{3-2\gamma_E}{4}(\delta+iz)^2
-\frac{3-2\gamma_E}{4}(\delta+iz),
\]
with \(G(0,0)=0\) defined by continuity. Then
\[
C(z,0)=\Sigma G(z,0)+C_1(z)+C_2(z).
\]
Accordingly, decompose \(A_{\mathrm{dig}}(t,1)=I_0(t)+I_1(t)+I_2(t)\), where
\[
I_0(t)\coloneqq \frac{t}{\pi}\int_0^\infty \Re\!\left[g_t(z)\frac{\Sigma G(z,0)}{iz}\right]dz,
\qquad
I_j(t)\coloneqq \frac{t}{\pi}\int_0^\infty \Re\!\left[g_t(z)\frac{C_j(z)}{iz}\right]dz,\quad j=1,2.
\]

For \(z>0\), with use of \(\log(iz)=\log z+i\pi/2\),
\[
\frac{\Sigma G(z,0)}{iz}
=
-\frac{3-2\gamma_E}{4}\Sigma
+\frac{\pi\Sigma}{4}z
-\frac{i\Sigma}{2}z\log z
+i\frac{3-2\gamma_E}{4}\Sigma z.
\]
Hence,
\[
I_0(t)
=
\frac{\Sigma t}{4}\int_0^\infty e^{-\Sigma t z^2/2}z\cos\!\left(\frac{\Sigma t}{2}z\right)dz
+R_0(t),
\]
where
\[
|R_0(t)|
\le
c\left(
t\int_0^\infty e^{-\Sigma t z^2/2}\,dz
+
t^2\int_0^\infty e^{-\Sigma t z^2/2}z^2\,dz
+
t^2\int_0^\infty e^{-\Sigma t z^2/2}z^2|\log z|\,dz
\right).
\]
After \(u=\sqrt{t}\,z\), \(R_0(t)=O(\sqrt{t}\,|\log t|)\). Also, since \(|\cos x-1|\le x^2/2\),
\[
\begin{aligned}
\frac{\Sigma t}{4}\int_0^\infty e^{-\Sigma t z^2/2}z\cos\!\left(\frac{\Sigma t}{2}z\right)dz
&=
\frac{\Sigma t}{4}\int_0^\infty e^{-\Sigma t z^2/2}z\,dz
+
O\!\left(
t^3\int_0^\infty e^{-\Sigma t z^2/2}z^3\,dz
\right)
\\
&=
\frac14+O(t),
\end{aligned}
\]
because \(\Sigma t\int_0^\infty e^{-\Sigma t z^2/2}z\,dz=1\). Therefore
\begin{equation}
I_0(t)=\frac14+O(\sqrt{t}\,|\log t|).
\end{equation}

By (G-4),
\[
\left|\frac{C_1(z)}{iz}\right|\le c(1+\log(2+z)),
\qquad
\sup_{z>0}\left|\frac{C_2(z)}{iz}\right|<\infty,
\]
so
\[
I_1(t)=O(\sqrt{t}\,|\log t|),
\qquad
I_2(t)=O(\sqrt{t}).
\]
Thus
\begin{equation}
A_{\mathrm{dig}}(t,1)=\frac14+O(\sqrt{t}\,|\log t|).
\end{equation}

Since \(e^{\Sigma t/8}=1+O(t)\), we conclude that
\[
A_{\mathrm{ATM}}(t)
=
-\frac{e^{\Sigma t/8}}{\sqrt{t}}
\left(
\sqrt{\frac{\pi}{2}}\,A_{\mathrm{call}}(t,1)
+
\sqrt{2\pi}\,A_{\mathrm{dig}}(t,1)
\right)
=
-\frac{\sqrt{2\pi}}{4}\frac{1}{\sqrt{t}}+O(|\log t|).
\]
\end{proof}

We next turn to the long-maturity limit.

\begin{proposition}\label{pr:a_atm_long}
Let \(F\) denote the extension of \(C(z,0)/(iz)\) given in Lemma C.2. Under (G-4), \(F\) admits an analytic extension to a neighborhood of \(-i/2\), and, as \(t\to\infty\),
\[
A_{\mathrm{ATM}}(t)
=-\frac{4}{\Sigma^{3/2}t}
\operatorname{Re}\left\{
F\left(-\frac{i}{2}\right)
-\frac{i}{2}F'\left(-\frac{i}{2}\right)
\right\}
+O(t^{-2}).
\]
In particular, \(A_{\mathrm{ATM}}(t)\to 0\).
\end{proposition}

\begin{proof}
Set
\[
G(z)\coloneqq \frac{C(z,1)}{-z^2+iz},
\qquad
A_{\mathrm{call}}(t,1)=\frac{t}{2\pi}\int_{\mathbb{R}} g_t(z-i)G(z)\,dz.
\]
By Lemma~\ref{le:analytic_lemma_call}, \(G\) extends continuously to \(\{z\in\mathbb{C}:0\le \Im z\le 1/2\}\), is analytic in its interior, and has at most polynomial growth there. Hence the contour can be shifted from \(\Im z=0\) to \(\Im z=1/2\), giving
\[
A_{\mathrm{call}}(t,1)
=
\frac{t}{2\pi}\int_{\mathbb{R}} g_t\!\left(x-\frac{i}{2}\right)G\!\left(x+\frac{i}{2}\right)\,dx.
\]
Since \(\left(x-\frac{i}{2}\right)^2+i\left(x-\frac{i}{2}\right)=x^2+\frac14\),
\[
\left|g_t\!\left(x-\frac{i}{2}\right)\right|
=
e^{-\Sigma t/8}e^{-\Sigma t x^2/2}.
\]
For the digital term, Lemma~\ref{le:analytic_lemma_dig} gives \(\overline{F(x)}=F(-x)\) for \(x\in\mathbb{R}\); hence
\[
A_{\mathrm{dig}}(t,1)
=
\frac{t}{2\pi}\int_{\mathbb{R}} g_t(x)F(x)\,dx.
\]
The same lemma also shows that \(F\) is analytic on \(\{-1/2<\Im z<0\}\), extends continuously to its closure, and has at most polynomial growth there. Thus the contour can be shifted to \(\Im z=-1/2\):
\[
A_{\mathrm{dig}}(t,1)
=
\frac{t}{2\pi}e^{-\Sigma t/8}
\int_{\mathbb{R}}e^{-\Sigma t x^2/2}F\!\left(x-\frac{i}{2}\right)\,dx.
\]
By (G-4), (D-1), and the argument in Lemma C.2, \(F\) is analytic in a neighborhood of \(-i/2\). Since \(C(z,\delta)\) depends only on \(\delta+iz\), we have \(G(x+i/2)=F(x-i/2)/(ix-1/2)\).
Thus, the above integral representations and Corollary 4.2 give
\[
A_{\mathrm{ATM}}(t)
=\sqrt{\frac{t}{2\pi}}
\int_{\mathbb{R}}e^{-\Sigma t x^2/2}
\operatorname{Re}\left\{
\frac{ix}{1/2-ix}F\left(x-\frac{i}{2}\right)
\right\}\,dx.
\]
Since \(F(-x-i/2)=\overline{F(x-i/2)}\), the real part in the integrand is even. Hence, Taylor expansion at \(x=0\) yields
\[
\operatorname{Re}\left\{
\frac{ix}{1/2-ix}F\left(x-\frac{i}{2}\right)
\right\}
=-4\operatorname{Re}\left\{
F\left(-\frac{i}{2}\right)
-\frac{i}{2}F'\left(-\frac{i}{2}\right)
\right\}x^2+O(x^4).
\]
By the polynomial growth bound in Lemma C.2, the contribution from \(|x|\ge\eta\) is exponentially small for every fixed \(\eta>0\). Gaussian integration therefore gives
\[
A_{\mathrm{ATM}}(t)
=-\frac{4}{\Sigma^{3/2}t}
\operatorname{Re}\left\{
F\left(-\frac{i}{2}\right)
-\frac{i}{2}F'\left(-\frac{i}{2}\right)
\right\}
+O(t^{-2}).
\]
\end{proof}

\subsection{Numerical experiment}

We illustrate the asymptotic formula in the FMLS case:
\[
q(r,-1)=1,\qquad q(r,1)=0,\qquad \rho(\{-1\})=1.
\]
This is the spectrally negative FMLS model of Carr and Wu \cite{carr2003finite}. Here \(\Sigma=1\), \(m=-1/2\), and
\begin{equation}
C(z,\delta)
=
-\frac{1}{2}(\delta+iz)^2\log(\delta+iz)
+\frac{3-2\gamma_E}{4}(\delta+iz)^2
-\frac{3-2\gamma_E}{4}(\delta+iz).
\end{equation}
We compute \(A_{\mathrm{ATM}}(t)\) from the Fourier representations in Corollary~4.2 by a fast-Fourier-transform-based discretization on a large symmetric interval with a uniform grid.

For comparison, we simulate \(X_t^\alpha\) for fixed \(\alpha\in(1,2)\) from the corresponding stable law. From the simulated ATM call and digital prices at strike \(1\), we recover the ATM implied volatility and then the ATM skew \(\partial_K\hat{\sigma}(t,K,\alpha)\rvert_{K=1}\) via the standard ATM identity.

Figure~\ref{fig:numerical_atmskew} shows that the first-order approximation \((2-\alpha)A_{\mathrm{ATM}}(t)\) becomes accurate as \(\alpha\uparrow2\), and that the log-log plot is consistent with Proposition~\ref{pr:a_atm_short}.
\begin{figure}[tbp]
    \centering
    \begin{minipage}[b]{0.49\columnwidth}
        \centering
        \includegraphics[width=\linewidth]{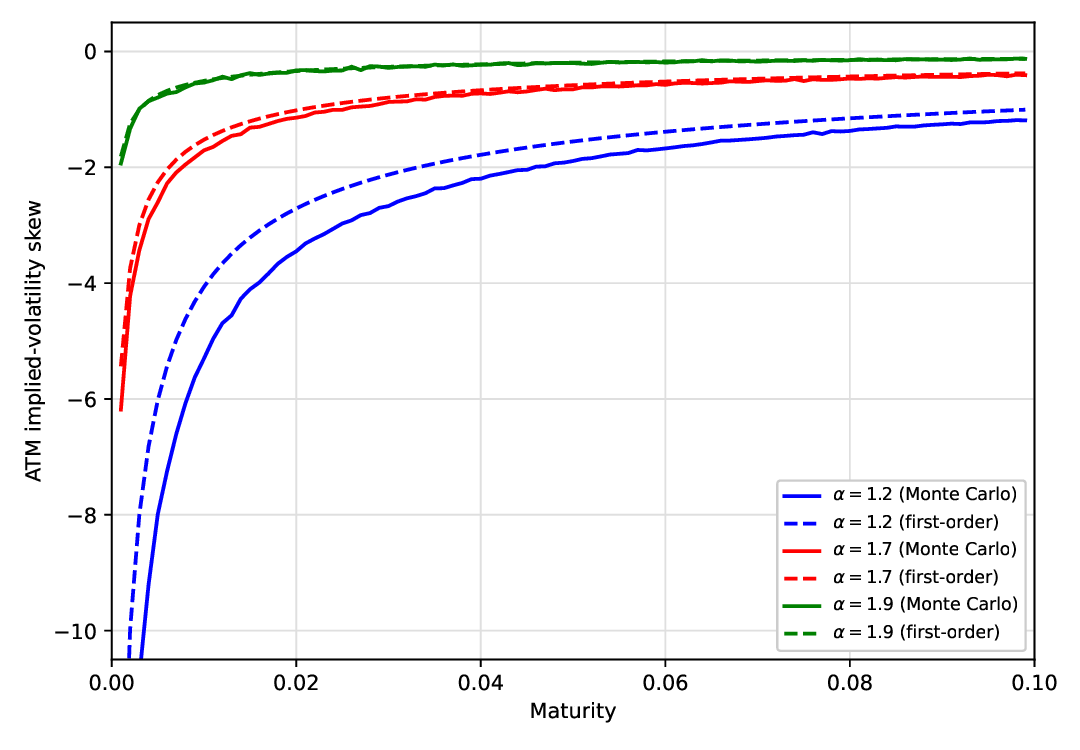}
        \subcaption{}
    \end{minipage}\hfill
    \begin{minipage}[b]{0.49\columnwidth}
        \centering
        \includegraphics[width=\linewidth]{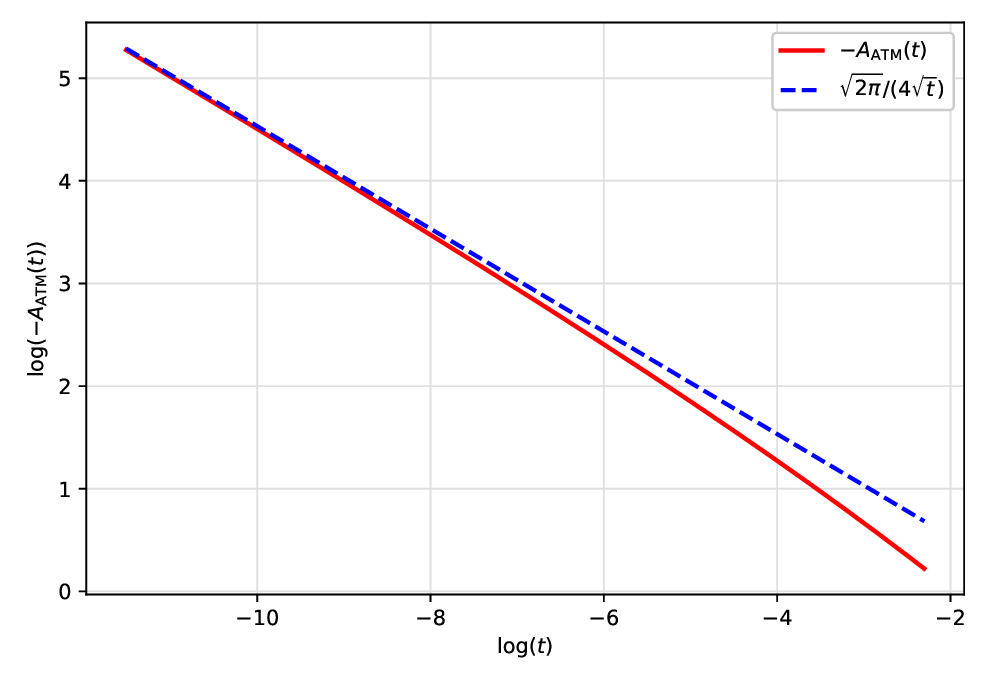}
        \subcaption{}
    \end{minipage}
    \caption{
        (a) ATM implied volatility skew \(\partial_K \hat{\sigma}(t,K,\alpha)\rvert_{K=1}\) from Monte Carlo simulation in the FMLS model (solid lines), and the approximation \((2-\alpha)A_{\mathrm{ATM}}(t)\) (dashed lines), where \(A_{\mathrm{ATM}}(t)\) is computed from the Fourier representation.
        (b) Log-log plot of \(-A_{\mathrm{ATM}}(t)\). The dashed line has slope \(-1/2\).
    }
    \label{fig:numerical_atmskew}
\end{figure}

\clearpage
\appendix

\section{Proof of Lemma 2.1}

\begin{proof}[Proof of Lemma 2.1]
Let $N$ be the jump measure of $X$, and define
\begin{equation}
    X^{(1,n)}(t,\omega)\coloneqq
    \int_{(0,t]\times\cbra{\abs{x}\ge 1/n}}x\,N(d sd x,\omega)
    -t\int_{\abs{x}\ge 1/n}x\cdot\bm{1}_{\cbra{\abs{x}\le 1}}\,\nu(d x),
\end{equation}
\begin{equation}
    X^{(1)}\coloneqq \lim_{n\to\infty}X^{(1,n)},\qquad
    X^{(0)}\coloneqq X-X^{(1)}.
\end{equation}
Then the L\'evy--It\^o decomposition \cite[Theorem 19.2]{ken-iti1999levy} says that
$X^{(0)}$ is a L\'evy process with triplet $(m,V,0)$ and
$X^{(1)}$ is a L\'evy process with triplet $(0,0,\nu)$.
We define
\begin{equation}
    [X]^n_t\coloneqq tV+\int_{(0,t]\times\cbra{\abs{x}\ge 1/n}} xx^{\top}N(d sd x)
\end{equation}
for each $n\in\mathbb{N}$. Of course, $[X]^n_t\to [X]_t$ holds as $n\to\infty$.
Fix $z\in\mathbb{R}^d$ and $U\in\mathbb{R}^{d\times d}$.
We have
\begin{align}
    \inner{z}{X_1}+\inner{U}{[X]_1}
    &=\inner{z}{X^{(0)}_1}+\inner{z}{X^{(1)}_1}+\inner{U}{[X]_1}\\
    &=\lim_{n\to \infty}\rbra*{\inner{z}{X^{(0)}_1}+\inner{z}{X^{(1,n)}_1}+\inner{U}{[X]^n_1}}.
\end{align}
Since $X^{(0)}_1$ and $N(A)\,(A\in\mathcal{B}((0,1]\times\left\{|x|\ge 1/n\right\}))$ are independent,
\begin{align}
    \Exp*{}{e^{i(\inner{z}{X_1}+\inner{U}{[X]_1})}}
    &=\lim_{n\to\infty}\Exp*{}{e^{i(\inner{z}{X^{(0)}_1}+\inner{z}{X^{(1,n)}_1}+\inner{U}{[X]^n_1})}}\\
    &=\Exp*{}{e^{i\inner{z}{X^{(0)}_1}}}\cdot\lim_{n\to\infty}\Exp*{}{e^{i(\inner{z}{X^{(1,n)}_1}+\inner{U}{[X]^n_1})}}.
\end{align}
By the definitions of $X^{(1,n)}$ and $[X]^n$, we obtain
\begin{align}
    \inner{z}{X^{(1,n)}_1}+\inner{U}{[X]^n_1}
    &=\int_{(0,1]\times\cbra{\abs{x}\ge 1/n}} (\inner{z}{x}+\inner{x}{Ux})N(d sd x)\\
    &\hspace{10mm}-\int_{\abs{x}\ge 1/n}\inner{z}{x}\bm{1}_{\left\{|x|\le 1\right\}}\nu(d x)+\inner{U}{V}.
\end{align}
Define $A^n_{ph}$ as
\begin{equation}
    A^n_{ph}\coloneqq
    \Set*{x\in\mathbb{R}^d}{\frac{h_j}{p}<x_j\le \frac{h_j+1}{p},\;j=1,\dots, d}
    \cap\cbra*{\abs*{x}\ge \frac{1}{n}}
\end{equation}
for each $p,n\in\mathbb{N}$ and $h=(h_j)_{j=1}^d\in\mathbb{Z}^d$. Therefore, we can get the limit
\begin{align}
    &\int_{(0,1]\times\cbra{\abs{x}\ge 1/n}} (\inner{z}{x}+\inner{x}{Ux})N(d sd x)\\
    &\hspace{20mm}=\lim_{p\to\infty}\sum_{h\in\mathbb{Z}^d}\rbra*{\inner{z}{h/p}+\inner{h/p}{Uh/p}}N((0,1]\times A^n_{ph}).
\end{align}
From the properties of the Poisson random measure $N$, 
\begin{align}
    &\phantom{=}\Exp*{}{\exp\cbra*{i\sum_{h\in\mathbb{Z}^d}\rbra*{\inner{z}{h/p}+\inner{h/p}{Uh/p}}N((0,1]\times A^n_{ph})}}\\
    &=\prod_{h\in\mathbb{Z}^d}\Exp*{}{\exp\cbra*{i\rbra*{\inner{z}{h/p}+\inner{h/p}{Uh/p}}N((0,1]\times A^n_{ph})}}\\
    &=\prod_{h\in\mathbb{Z}^d}\exp\cbra*{\rbra*{e^{i\rbra*{\inner{z}{h/p}+\inner{h/p}{Uh/p}}}-1}\nu(A^n_{ph})}\\
    &=\exp\cbra*{\sum_{h\in\mathbb{Z}^d}\rbra*{e^{i\rbra*{\inner{z}{h/p}+\inner{h/p}{Uh/p}}}-1}\nu(A^n_{ph})}
\end{align}
holds. Since $\nu$ is a finite measure on any subsets excluding a neighborhood of $0$,
Lebesgue's dominated convergence theorem shows the convergence
\begin{equation}
    \int_{|x|\ge 1/n}(e^{i(\inner*{z}{x}+\inner{x}{Ux})}-1)\nu(d x)=
    \lim_{p\to\infty}\sum_{h\in\mathbb{Z}^d}\rbra*{e^{i\rbra*{\inner{z}{h/p}+\inner{h/p}{Uh/p}}}-1}\nu(A^n_{ph}).
\end{equation}
Then
\begin{align}
    &\Exp*{}{\exp\rbra*{i\int_{(0,1]\times\cbra{\abs{x}\ge 1/n}} (\inner{z}{x}+\inner{x}{Ux})N(d sd x)}}\\
    &\hspace{30mm}=\exp\cbra*{\int_{|x|\ge 1/n}(e^{i(\inner{z}{x}+\inner{x}{Ux})}-1)\nu(d x)},
\end{align}
and we get
\begin{equation}
    \Exp*{}{e^{i(\inner{z}{X^{(1,n)}_1}+\inner{U}{[X]^n_1})}}
    =\exp\cbra*{i\inner{U}{V}+\int_{|x|\ge 1/n}(e^{i(\inner{z}{x}+\inner{x}{Ux})}-1-i\inner{z}{x}\bm{1}_{\left\{|x|\le 1\right\}})\nu(d x)}.
\end{equation}
When $|x|\le 1$, 
\begin{align}
    |e^{i(\inner{z}{x}+\inner{x}{Ux})}-1-i\inner{z}{x}|
    &\le |e^{i(\inner{z}{x}+\inner{x}{Ux})}-1-i(\inner{z}{x}+\inner{x}{Ux})|+|\inner{x}{Ux}|\\
    &\le \frac{1}{2}|\inner{z}{x}+\inner{x}{Ux}|^2+\norm*{U}|x|^2\\
    &\le \rbra*{\frac{1}{2}(|z|+\norm*{U})^2+\norm*{U}}|x|^2
\end{align}
holds. With $\int_{\mathbb{R}^d}(|x|^2\land 1)\nu(d x)<\infty$,
it follows from Lebesgue's dominated convergence theorem that
\begin{equation}
    \Exp*{}{e^{i(\inner{z}{X^{(1)}_1}+\inner{U}{[X]_1})}}
    =\exp\cbra*{i\inner{U}{V}+\int_{\mathbb{R}^d}(e^{i(\inner{z}{x}+\inner{x}{Ux})}-1-i\inner{z}{x}\bm{1}_{\left\{|x|\le 1\right\}})\nu(d x)},
\end{equation}
where $n\to\infty$.
Summarizing the above, we have the characteristic function of the joint distribution as
\begin{align}
    &\Exp*{}{e^{i(\inner{z}{X_1}+\inner{U}{[X]_1})}}\\
    &\hspace{10mm}=\Exp*{}{e^{i(\inner{z}{X^{(0)}_1}+\inner{z}{X^{(1)}_1}+\inner{U}{[X]_1})}}\\
    &\hspace{10mm}=\exp\cbra*{i\inner{m}{z}-\frac{1}{2}\inner*{z}{Vz}+i\inner{U}{V}+\int_{\mathbb{R}^d}(e^{i(\inner{z}{x}+\inner{x}{Ux})}-1-i\inner{z}{x}\bm{1}_{\left\{|x|\le 1\right\}})\nu(d x)}.
\end{align}
This is the end of the proof.
\end{proof}

\section{Proof of Lemma 4.1}

We first prove a decay estimate for a one-dimensional GTS process.

\begin{proposition}[Decay of the characteristic function]\label{pr:decayspeed}
Let \(\alpha\in(1,2)\), \(\delta_0>0\), and \(X\) be an \(\mathbb{R}\)-valued GTS process with parameters \((m,\alpha,\rho,q)\) satisfying \((\mathrm{G}\text{-}1)\), \((\mathrm{G}\text{-}2)\), \((\mathrm{D}\text{-}1)\), and \((\mathrm{D}\text{-}2)\). Let \(g_t\) be the characteristic function of \(X_t\), and fix \(\delta\in[0,\delta_0]\). Then there exist \(\gamma_\alpha>0\), \(C_\alpha\in\mathbb{R}\), and \(z_0>0\) such that
\[
|g_t(z-i\delta)|\le \exp(-t\gamma_\alpha |z|^\alpha+tC_\alpha),\qquad |z|\ge z_0.
\]
\end{proposition}

\begin{proof}
Since \(S=\{-1,1\}\),
\[
|g_t(z-i\delta)|
=
\exp\!\left(
tm\delta+t\rho(\{1\}) I_+(z,\delta)+t\rho(\{-1\}) I_-(z,\delta)
\right),
\]
where
\[
I_+(z,\delta)\coloneqq \int_0^\infty \left(e^{\delta r}\cos(zr)-1-\delta r\bm{1}_{(0,1)}(r)\right)\frac{q(r,1)}{r^{\alpha+1}}\,dr
\]
and
\[
I_-(z,\delta)\coloneqq \int_0^\infty \left(e^{-\delta r}\cos(zr)-1+\delta r\bm{1}_{(0,1)}(r)\right)\frac{q(r,-1)}{r^{\alpha+1}}\,dr.
\]

We first estimate \(I_+(z,\delta)\). On \((0,1)\), \(\cos(zr)\le 1\) and \(e^{\delta r}-1-\delta r\le c_\delta r^2\), so by \((\mathrm{G}\text{-}1)\),
\[
\int_0^1 \left(e^{\delta r}\cos(zr)-1-\delta r\right)\frac{q(r,1)}{r^{\alpha+1}}\,dr
\le
Mc_\delta \int_0^1 r^{1-\alpha}\,dr.
\]
On \([1,\infty)\), \((\mathrm{D}\text{-}1)\) gives \(q(r,1)\le M_0 e^{-\delta_0 r}\) for some \(M_0>0\); hence
\[
\int_1^\infty \left(e^{\delta r}\cos(zr)-1\right)\frac{q(r,1)}{r^{\alpha+1}}\,dr
\le
M_0\int_1^\infty \frac{e^{-(\delta_0-\delta)r}}{r^{\alpha+1}}\,dr
\le
M_0\int_1^\infty r^{-\alpha-1}\,dr.
\]
Therefore
\begin{equation}\label{eq:Iplus-bound}
I_+(z,\delta)\le C_{\alpha,\delta}^{(1)}
\end{equation}
for some constant \(C_{\alpha,\delta}^{(1)}\).

Next we estimate \(I_-(z,\delta)\). Since \(e^{-\delta r}\cos(zr)-1\le0\), the contribution of \([1,\infty)\) is non-positive, so
\[
I_-(z,\delta)\le \int_0^1 \left(e^{-\delta r}\cos(zr)-1+\delta r\right)\frac{q(r,-1)}{r^{\alpha+1}}\,dr.
\]
For \(r\in(0,1)\),
\[
e^{-\delta r}\cos(zr)-1+\delta r
=
(\cos(zr)-1)+\delta r(1-\cos(zr))+(e^{-\delta r}-1+\delta r)\cos(zr);
\]
hence, using \(|\cos|\le1\) and \(e^{-\delta r}-1+\delta r\le c_\delta r^2\), we obtain
\[
I_-(z,\delta)\le J_1(z)+\delta J_2(z)+Mc_\delta \int_0^1 r^{1-\alpha}\,dr,
\]
where
\[
J_1(z)\coloneqq \int_0^1 (\cos(zr)-1)\frac{q(r,-1)}{r^{\alpha+1}}\,dr,
\qquad
J_2(z)\coloneqq \int_0^1 (1-\cos(zr))\frac{q(r,-1)}{r^\alpha}\,dr.
\]

By \((\mathrm{D}\text{-}2)\) and \((\mathrm{G}\text{-}2)\), there exist \(c_->0\) and \(r_0\in(0,1)\) such that \(q(r,-1)\ge c_-\) on \((0,r_0)\). Therefore
\[
J_1(z)\le c_-\int_0^{r_0} (\cos(zr)-1)r^{-\alpha-1}\,dr.
\]
After \(u=|z|r\),
\[
J_1(z)\le c_-|z|^\alpha \int_0^{r_0|z|} (\cos u-1)u^{-\alpha-1}\,du.
\]
Since
\[
\kappa_\alpha\coloneqq -\int_0^\infty (\cos u-1)u^{-\alpha-1}\,du
=
-\Gamma(-\alpha)\cos\frac{\pi\alpha}{2}>0
\]
by \cite[Lemma 14.11]{ken-iti1999levy}, there exists \(z_1>0\) such that
\[
\int_0^{r_0|z|} (\cos u-1)u^{-\alpha-1}\,du \le -\frac{\kappa_\alpha}{2},
\qquad |z|\ge z_1.
\]
Hence
\[
J_1(z)\le -\frac{c_-\kappa_\alpha}{2}|z|^\alpha,\qquad |z|\ge z_1.
\]

Also, by \((\mathrm{G}\text{-}1)\),
\[
J_2(z)\le M\int_0^1 (1-\cos(zr))r^{-\alpha}\,dr
=
M|z|^{\alpha-1}\int_0^{|z|} (1-\cos u)u^{-\alpha}\,du.
\]
Since
\[
L_\alpha\coloneqq \int_0^\infty (1-\cos u)u^{-\alpha}\,du<\infty,
\]
we get
\[
J_2(z)\le ML_\alpha |z|^{\alpha-1}.
\]
Therefore, for \(|z|\ge z_1\),
\[
I_-(z,\delta)\le -\frac{c_-\kappa_\alpha}{2}|z|^\alpha+\delta ML_\alpha |z|^{\alpha-1}+C_{\alpha,\delta}^{(2)}
\]
for some constant \(C_{\alpha,\delta}^{(2)}\). Hence there exists \(z_0\ge z_1\) such that
\begin{equation}\label{eq:Iminus-bound}
I_-(z,\delta)\le -\frac{c_-\kappa_\alpha}{4}|z|^\alpha+C_{\alpha,\delta}^{(2)},\qquad |z|\ge z_0.
\end{equation}

Combining \eqref{eq:Iplus-bound} and \eqref{eq:Iminus-bound}, we obtain
\[
\log |g_t(z-i\delta)|
\le
-t\gamma_\alpha |z|^\alpha+tC_\alpha,\qquad |z|\ge z_0,
\]
where
\[
\gamma_\alpha\coloneqq \frac14 \rho(\{-1\}) c_-\kappa_\alpha>0,
\qquad
C_\alpha\coloneqq m\delta+\rho(\{1\}) C_{\alpha,\delta}^{(1)}+\rho(\{-1\}) C_{\alpha,\delta}^{(2)}.
\]
This proves the claim.
\end{proof}

\begin{proof}[Proof of Lemma 4.1]
Let \(X^\alpha\) be the scaled GTS process. Its L\'evy measure is
\[
\nu_\alpha(B)
=
(2-\alpha)\int_S \rho(d\xi)\int_0^\infty \bm{1}_B(r\xi)\frac{q(r,\xi)}{r^{\alpha+1}}\,dr.
\]
Fix \(c_->0\) and \(r_0\in(0,1)\) so that \(q(r,-1)\ge c_-\) on \((0,r_0)\). Applying Proposition~\ref{pr:decayspeed} to the tempering function \((2-\alpha)q\), we obtain
\[
|g_t^\alpha(z-i\delta)|\le \exp(-t\gamma_\alpha'|z|^\alpha+tC_\alpha'),\qquad |z|\ge z_{0,\alpha},
\]
where
\[
\gamma_\alpha'=\frac14 (2-\alpha)\rho(\{-1\})c_-\kappa_\alpha
\]
and \(C_\alpha'=m_\alpha\delta+R_{\alpha,\delta}\). Here \(R_{\alpha,\delta}\) is bounded above uniformly in \(\alpha\in(3/2,2)\), because in the proof above the only singular positive term is \(\int_0^1 r^{1-\alpha}\,dr=(2-\alpha)^{-1}\), and this term is multiplied by \(2-\alpha\).

Moreover,
\[
(2-\alpha)\kappa_\alpha
=
-(2-\alpha)\Gamma(-\alpha)\cos\frac{\pi\alpha}{2}\longrightarrow \frac12
\qquad (\alpha\uparrow2)
\]
by \cite[Lemma 14.11]{ken-iti1999levy}. Hence there exist \(\alpha_0\in(3/2,2)\) and \(c_0>0\) such that
\[
(2-\alpha)\kappa_\alpha\ge c_0,
\qquad
\alpha\in(\alpha_0,2).
\]
Therefore
\[
\gamma_\alpha'\ge \frac14 \rho(\{-1\})c_-c_0,
\qquad
\alpha\in(\alpha_0,2).
\]
Since \(m_\alpha\to -\Sigma/2\), there exists \(C'\in\mathbb{R}\) such that
\[
C_\alpha'\le C',
\qquad
\alpha\in(\alpha_0,2).
\]

Finally, in the proof above, \(z_{0,\alpha}\) is chosen so that
\[
\delta ML_\alpha |z|^{\alpha-1}\le \frac{c_-\kappa_\alpha}{4}|z|^\alpha.
\]
Since \(\sup_{\alpha\in(3/2,2)}L_\alpha<\infty\) and \((2-\alpha)\kappa_\alpha\ge c_0\) on \((\alpha_0,2)\), there exists \(z_0>0\) such that \(z_{0,\alpha}\le z_0\) for all \(\alpha\in(\alpha_0,2)\). Thus,
\[
|g_t^\alpha(z-i\delta)|\le \exp(-t\gamma'|z|^\alpha+tC'),
\qquad
\alpha\in(\alpha_0,2),\quad |z|\ge z_0,
\]
where $\gamma'\coloneqq \frac14 \rho(\{-1\})c_-c_0>0$.
\end{proof}

\section{Analytic continuation arguments}

Throughout this appendix, we assume condition \((\mathrm{G}\text{-}4)\).

\begin{lemma}\label{le:analytic_lemma_call}
Define \(G(z)\coloneqq C(z,1)/(-z^2+iz)\) for \(z\in\mathbb{R}\setminus\{0\}\). Then \(G\) admits a continuous extension to
\[
S_+\coloneqq \Set*{z\in\mathbb{C}}{0\le \Im z\le \frac12},
\]
analytic on \(S_+^\circ\coloneqq \Set*{z\in\mathbb{C}}{0<\Im z<\frac12}\). Moreover, there exists a polynomial \(P_2\) of degree at most \(2\) such that
\[
|G(x+iy)|\le P_2(1+|x|),\qquad x\in\mathbb{R},\quad 0\le y\le \frac12,
\]
and a constant \(c>0\) such that
\[
|G(x)|\le c\bigl(1+\log(2+|x|)\bigr),\qquad x\in\mathbb{R}.
\]
\end{lemma}

\begin{proof}
As in the proof of Theorem~4.1, let
\begin{align}
C_0(z)&\coloneqq
\Sigma\left(
-\frac12(1+iz)^2\log(1+iz)
+\frac{3-2\gamma_E}{4}(1+iz)^2
-\frac{3-2\gamma_E}{4}(1+iz)
\right), \\
J_1(z)&\coloneqq
\int_S \rho(d\xi)\, q_r(0+,\xi)\int_0^1
\left(e^{(1+iz)r\xi}-1-(1+iz)r\xi\right)\frac{dr}{r^2}, \\
J_2(z)&\coloneqq
\int_S \rho(d\xi)\int_0^\infty
\left(e^{(1+iz)r\xi}-1-(1+iz)r\xi \bm{1}_{(0,1)}(r)\right)\frac{\Delta q(r,\xi)}{r^3}\,dr,
\end{align}
and
\[
\widetilde C(z)\coloneqq C_0(z)+J_1(z)+J_2(z)-(1+iz)\bigl(J_1(0)+J_2(0)\bigr).
\]

Choose \(\eta>0\) so that \(1+\eta<\delta_0\), and set
\[
V_\eta\coloneqq \Set*{z\in\mathbb{C}}{|z|<\frac12,\ |\Im z|<\eta}.
\]
Since \(\Re(1+iz)>0\) on \(V_\eta\), \(C_0\) is analytic there. Let \(K\subset V_\eta\) be compact. By \((\mathrm{G}\text{-}4)\), \(q(0+,1)=0\), and \(q(r,1)=O(e^{-\delta_0 r})\), the integrands defining \(J_1\) and \(J_2\), together with their first \(z\)-derivatives, are dominated on \(K\) by integrable functions: \(c_K|q_r(0+,\xi)|\) for \(J_1\); \(c_K |\Delta q(r,\xi)|/r^2\) on \((0,1)\times S\); \(c_K |\Delta q(r,-1)|/r^2\) on \([1,\infty)\times\{-1\}\); and \(c_K e^{(1+\eta)r}q(r,1)/r^2\) on \([1,\infty)\times\{1\}\). Hence \(J_1\), \(J_2\), and \(\widetilde C\) are analytic on \(V_\eta\).

Since \(C_0(0)=\widetilde C(0)=0\), there exists an analytic function \(H\) on \(V_\eta\) such that \(\widetilde C(z)=zH(z)\). As \(i-z\neq0\) on \(V_\eta\), \(\widehat G(z)\coloneqq H(z)/(i-z)\) is analytic on \(V_\eta\), and \(\widehat G(z)=\widetilde C(z)/(-z^2+iz)\) on \(V_\eta\setminus\{0\}\).

We next work on \(S_+\). Since \(\Re(1+iz)\ge 1/2\) on \(S_+\), \(C_0\) is continuous on \(S_+\) and analytic on \(S_+^\circ\), and
\[
|C_0(x+iy)|\le c(1+|x|)^2(1+\log(2+|x|)),\qquad x\in\mathbb{R},\quad 0\le y\le \frac12.
\]
Similarly, \((\mathrm{G}\text{-}4)\), \(q(0+,1)=0\), and the same tail assumption yield
\[
|J_1(x+iy)|+|J_2(x+iy)|\le c(1+|x|)^2,\qquad x\in\mathbb{R},\quad 0\le y\le \frac12.
\]
Dominated convergence yields continuity of \(J_1\) and \(J_2\) on \(S_+\), and differentiation under the integral sign on compact subsets of \(S_+^\circ\) yields analyticity; only \(\xi=1\), \(r\ge1\) needs comment, where we use \(c_K e^{(1-\varepsilon_K)r}q(r,1)/r^2\) for some \(\varepsilon_K\in(0,1/2)\). Therefore \(\widetilde C\) is continuous on \(S_+\) and analytic on \(S_+^\circ\), and
\[
|\widetilde C(x+iy)|\le c(1+|x|)^2(1+\log(2+|x|)),\qquad x\in\mathbb{R},\quad 0\le y\le \frac12.
\]

Define \(\widetilde G(z)\coloneqq \widetilde C(z)/(-z^2+iz)\) on \(S_+\setminus\{0\}\). On \((S_+\cap V_\eta)\setminus\{0\}\), this agrees with \(\widehat G\), so \(\widetilde G\) extends continuously through \(0\) by our setting \(\widetilde G(0)\coloneqq \widehat G(0)\). Since \(\widetilde C(x)=C(x,1)\) on \(\mathbb{R}\), we have \(\widetilde G(x)=G(x)\) for \(x\in\mathbb{R}\).

Because \(\widetilde G\) is continuous on \(\Set*{x+iy\in\mathbb{C}}{|x|\le1,\ 0\le y\le \frac12}\), it is bounded there. If \(|x|>1\) and \(z=x+iy\in S_+\) then \(|-z^2+iz|\ge c|x|^2\), so
\[
|\widetilde G(x+iy)|
\le
c\,\frac{(1+|x|)^2(1+\log(2+|x|))}{|x|^2}
\le c\bigl(1+\log(2+|x|)\bigr).
\]
Hence, \(|\widetilde G(x+iy)|\le P_2(1+|x|)\) for some polynomial \(P_2\) of degree at most \(2\). On the real axis, \(|-x^2+ix|=|x|\sqrt{1+x^2}\), so the same estimate gives the sharper bound
\[
|G(x)|\le c\bigl(1+\log(2+|x|)\bigr),\qquad x\in\mathbb{R}.
\]
\end{proof}

\begin{lemma}\label{le:analytic_lemma_dig}
Define \(F(z)\coloneqq C(z,0)/(iz)\) for \(z\in\mathbb{R}\setminus\{0\}\). Then \(F\) admits a continuous extension to
\[
S_-\coloneqq \Set*{z\in\mathbb{C}}{-\frac12\le \Im z\le 0},
\]
analytic on \(S_-^\circ\coloneqq \Set*{z\in\mathbb{C}}{-\frac12<\Im z<0}\). Moreover, there exists a polynomial \(P_2\) of degree at most \(2\) such that
\[
|F(x+iy)|\le P_2(1+|x|),\qquad x\in\mathbb{R},\quad -\frac12\le y\le 0.
\]
Finally,
\[
\overline{F(z)}=F(-\overline z),\qquad z\in S_-.
\]
\end{lemma}

\begin{proof}
The proof is parallel to the proof of Lemma~\ref{le:analytic_lemma_call}. Let
\begin{align}
C_0(z)&\coloneqq
\Sigma\left(
-\frac12(iz)^2\log(iz)
+\frac{3-2\gamma_E}{4}(iz)^2
-\frac{3-2\gamma_E}{4}iz
\right), \\
K_1&\coloneqq
\int_S \rho(d\xi)\, q_r(0+,\xi)\int_0^1
\left(e^{r\xi}-1-r\xi\right)\frac{dr}{r^2}, \\
K_2&\coloneqq
\int_S \rho(d\xi)\int_0^\infty
\left(e^{r\xi}-1-r\xi\bm{1}_{(0,1)}(r)\right)\frac{\Delta q(r,\xi)}{r^3}\,dr, \\
C_1(z)&\coloneqq
\int_S \rho(d\xi)\, q_r(0+,\xi)\int_0^1
\left(e^{izr\xi}-1-izr\xi\right)\frac{dr}{r^2}
-izK_1, \\
C_{2,0}(z)&\coloneqq
\int_S \rho(d\xi)\int_0^1
\left(e^{izr\xi}-1-izr\xi\right)\frac{\Delta q(r,\xi)}{r^3}\,dr, \\
C_{2,\infty}(z)&\coloneqq
\int_S \rho(d\xi)\int_1^\infty
\left(e^{izr\xi}-1\right)\frac{\Delta q(r,\xi)}{r^3}\,dr,
\end{align}
and
\[
\widetilde C(z)\coloneqq C_0(z)+C_1(z)+C_{2,0}(z)+C_{2,\infty}(z)-izK_2.
\]

Since \(\Re(iz)\ge0\) on \(S_-\), \(C_0\) is continuous on \(S_-\) and analytic on \(S_-^\circ\), and
\[
|C_0(x+iy)|\le c(1+|x|)^2(1+\log(2+|x|)),\qquad x\in\mathbb{R},\quad -\frac12\le y\le 0.
\]
For \(z=x+iy\in S_-\), the same Taylor remainder estimate as in Lemma~\ref{le:analytic_lemma_call} gives
\[
|e^{izr\xi}-1-izr\xi|\le c(1+|x|)^2r^2e^{r/2}\qquad \text{on } (0,1)\times S;
\]
hence \((\mathrm{G}\text{-}4)\) implies
\[
|C_1(x+iy)|+|C_{2,0}(x+iy)|\le c(1+|x|)^2,\qquad x\in\mathbb{R},\quad -\frac12\le y\le 0.
\]
For \(C_{2,\infty}\), if \(\xi=1\) then \(|e^{izr}-1|\le 1+e^{r/2}\), whereas if \(\xi=-1\) then \(|e^{-izr}-1|\le2\). Using \(q(0+,1)=0\), \(\Delta q(r,1)=q(r,1)\) for \(r\ge1\), and the exponential tail of \(q(r,1)\), we obtain
\[
|C_{2,\infty}(x+iy)|\le c,\qquad x\in\mathbb{R},\quad -\frac12\le y\le 0.
\]

These bounds yield continuity of \(C_1\), \(C_{2,0}\), and \(C_{2,\infty}\) on \(S_-\) by dominated convergence. On each compact \(K\subset S_-^\circ\), their first \(z\)-derivatives are dominated by integrable functions; again the only non-trivial part is \(\xi=1\), \(r\ge1\), where we use \(e^{(1/2-\varepsilon_K)r}q(r,1)/r^2\) for some \(\varepsilon_K\in(0,1/2)\). Hence they are analytic on \(S_-^\circ\), and so is \(\widetilde C\). Moreover,
\[
|\widetilde C(x+iy)|\le c(1+|x|)^2(1+\log(2+|x|)),\qquad x\in\mathbb{R},\quad -\frac12\le y\le 0.
\]

We next treat \(z=0\). Since \(iz\log(iz)\to0\) as \(z\to0\) in \(S_-\), \(C_0(z)/(iz)\) has a finite limit at \(0\). Retaining the factor \(|z|^2\) in the preceding Taylor estimate gives
\[
C_1(z)=-izK_1+O(|z|^2),\qquad C_{2,0}(z)=O(|z|^2),\qquad z\to0 \text{ in } S_-.
\]
Also,
\[
\frac{e^u-1}{u}=\int_0^1 e^{\theta u}\,d\theta,
\]
so \(C_{2,\infty}(z)/(iz)\) is dominated on \(S_-\setminus\{0\}\) by
\[
\bm{1}_{\{\xi=1\}}e^{r/2}\frac{|\Delta q(r,1)|}{r^2}
+
\bm{1}_{\{\xi=-1\}}\frac{|\Delta q(r,-1)|}{r^2},
\]
which is integrable. Thus each term in \(\widetilde C(z)/(iz)\) has a finite limit at \(0\), so \(F\) extends continuously to \(0\); defining \(F(0)\) by that limit, we obtain a continuous extension \(F\) to \(S_-\), analytic on \(S_-^\circ\).

Because \(F\) is continuous on \(\Set*{x+iy\in\mathbb{C}}{|x|\le1,\ -\frac12\le y\le0}\), it is bounded there. For \(|x|>1\) and \(z=x+iy\in S_-\), \(|z|\ge |x|\); hence
\[
|F(x+iy)|
\le
c\,\frac{(1+|x|)^2(1+\log(2+|x|))}{|x|}
\le P_2(1+|x|),
\]
for some polynomial \(P_2\) of degree at most \(2\).

Finally, if \(z\in S_-^\circ\) then \(iz\) lies in the open right half-plane, so \(\overline{\log(iz)}=\log(i(-\overline z))\). Since all coefficients are real, \(\overline{\widetilde C(z)}=\widetilde C(-\overline z)\), and therefore
\[
\overline{F(z)}
=
\overline{\frac{\widetilde C(z)}{iz}}
=
\frac{\widetilde C(-\overline z)}{i(-\overline z)}
=
F(-\overline z).
\]
By continuity, this identity extends to all \(z\in S_-\).
\end{proof}

\section*{Acknowledgements}
M. Fukasawa is grateful to Martin Forde for stimulating discussions that in particular brought him  an idea of this research.

\printbibliography

\end{document}